\documentclass{article}%
\usepackage{graphicx}
\usepackage{amsmath}
\usepackage{amsfonts}
\usepackage{amssymb}%
\providecommand{\U}[1]{\protect\rule{.1in}{.1in}}
\providecommand{\U}[1]{\protect\rule{.1in}{.1in}}
\providecommand{\U}[1]{\protect\rule{.1in}{.1in}}
\providecommand{\U}[1]{\protect\rule{.1in}{.1in}}
\providecommand{\U}[1]{\protect\rule{.1in}{.1in}}
\providecommand{\U}[1]{\protect\rule{.1in}{.1in}}
\providecommand{\U}[1]{\protect\rule{.1in}{.1in}}
\providecommand{\U}[1]{\protect\rule{.1in}{.1in}}
\providecommand{\U}[1]{\protect\rule{.1in}{.1in}}
\providecommand{\U}[1]{\protect\rule{.1in}{.1in}}
\providecommand{\U}[1]{\protect\rule{.1in}{.1in}}
\providecommand{\U}[1]{\protect\rule{.1in}{.1in}}
\providecommand{\U}[1]{\protect\rule{.1in}{.1in}}
\providecommand{\U}[1]{\protect\rule{.1in}{.1in}}
\providecommand{\U}[1]{\protect\rule{.1in}{.1in}}
\providecommand{\U}[1]{\protect\rule{.1in}{.1in}}
\providecommand{\U}[1]{\protect\rule{.1in}{.1in}}
\providecommand{\U}[1]{\protect\rule{.1in}{.1in}}
\providecommand{\U}[1]{\protect\rule{.1in}{.1in}}
\providecommand{\U}[1]{\protect\rule{.1in}{.1in}}
\providecommand{\U}[1]{\protect\rule{.1in}{.1in}}
\providecommand{\U}[1]{\protect\rule{.1in}{.1in}}
\providecommand{\U}[1]{\protect\rule{.1in}{.1in}}
\providecommand{\U}[1]{\protect\rule{.1in}{.1in}}
\providecommand{\U}[1]{\protect\rule{.1in}{.1in}}
\providecommand{\U}[1]{\protect\rule{.1in}{.1in}}
\providecommand{\U}[1]{\protect\rule{.1in}{.1in}}
\providecommand{\U}[1]{\protect\rule{.1in}{.1in}}
\providecommand{\U}[1]{\protect\rule{.1in}{.1in}}

\newtheorem{theorem}{Theorem}
{}

\newtheorem{conclusion}{Conclusion}
\newtheorem{condition}{Condition}

\newtheorem{definition}{Definition}

\newtheorem{lemma}{Lemma}
{}

\newtheorem{proposition}{Proposition}
\newtheorem{remark}{Remark}

\newenvironment{proof}[1][Proof]{\textbf{#1.} }{\ \rule{0.5em}{0.5em}}

\begin{document}

\title{Spectral Expansion Series with Parenthesis for the Nonself-adjoint Periodic
Differential Operators}
\author{O. A. Veliev\\{\small \ Depart. of Math., Dogus University, }\\{\small Ac\i badem, 34722, Kadik\"{o}y, \ Istanbul, Turkey.}\\\ {\small e-mail: oveliev@dogus.edu.tr}}
\date{}
\maketitle

\begin{abstract}
In this paper we construct the spectral expansion for the differential
operator generated in $L_{2}(-\infty,\infty)$ by ordinary differential
expression of arbitrary order with periodic complex-valued coefficients by
introducing new concepts as essential spectral singularities and singular
quasimomenta and using the series with parenthesis. Moreover, we find a
criteria for which the spectral expansion coincides with the Gelfand expansion
for the self-adjoint case.



\end{abstract}

\section{ Introduction and Preliminary Facts}

We consider the operator $L$ generated in the space $L_{2}(-\infty,\infty)$ by
the differential expression
\begin{equation}
\ell(y)=y^{(n)}+p_{2}y^{(n-2)}+\ldots+p_{n}y,
\end{equation}
where $p_{2},\;p_{3},\;\ldots,\;p_{n}$ are \ the $1$-periodic, complex-valued
functions satisfying
\begin{equation}
p_{k}^{(n-k)}\in L_{1}[0,1],\text{ }\forall k=2,3,...,n.
\end{equation}
The spectrum $\sigma(L)$ of $L$ is the union of the spectra $\sigma(L_{t})$ of
the operators $L_{t}$ for $t\in\lbrack0,2\pi)$ (see [2, 4, 7]) generated in
$L_{2}[0,1]$ by (1) and the $t$-periodic boundary conditions%

\begin{equation}
U_{\nu}(y):=y^{(\nu-1)}(1)-e^{it}y^{(\nu-1)}(0)=0,\hspace{1cm}\nu
=1,2,\cdots,n.
\end{equation}
The spectral expansion for the self-adjoint operator $L$\ was constructed by
Gelfand [2], Titchmarsh [8] and Tkachenko [9]. The existence of the spectral
singularities and the absence of the Parseval's equality for the
nonself-adjoint operator $L_{t}$ do not allow us to apply the elegant method
of Gelfand (see\ [2]) for construction of the spectral expansion for the
nonself-adjoin operator $L$. Note that the spectral singularity of $L$ is the
point of $\sigma(L)$ in neighborhood on which the projections of the operator
$L$ are not uniformly bounded. Therefore the existence of the spectral
singularities does not allow $L$ to be the spectral operators and hence the
theory of spectral operators is ineffective for the construction of the
spectral expansion for the nonself-adjoint periodic differential operators.
This situation give rise to find out a new approach for the construction of
the spectral expansion of $L$.

In this paper, we construct the spectral expansion for the nonself-adjoint
operator $L$ by introducing new concepts as singular quasimomenta (SQ),
essential spectral singularities (ESS) and ESS at infinity defined at the end
of this section. The case $n=2$ was investigated in [16]. To discuss in detail
the obtained results we need some preliminary facts about:

\textbf{(a) }eigenvalues and eigenfunctions of $L_{t},$

\textbf{(b)} spectral singularities and spectrality of $L$,

\textbf{(c) }spectral expansion of $L.$

\textbf{(a) Eigenvalues and eigenfunctions of }$L_{t}.$ First let us introduce
some definitions, notations and results of [12] used in this paper. It is
well-known [12] that the boundary conditions (3) \ \ is strongly regular for
$t\in\mathbb{C}$ and for $t\in\left(  \mathbb{C}\backslash\left\{  n\pi
:n\in\mathbb{Z}\right\}  \right)  $ if $n$ is odd $\left(  n=2\mu+1\right)  $
and even $\left(  n=2\mu\right)  $ respectively, where $\mu=1,2,...$.
Therefore the system of the root functions of $L_{t}$ for $t\neq0,\pi$ if
$n=2\mu$ and for $t\in\mathbb{C}$ if $n=2\mu+1$ form a Riesz basis in
$L_{2}(0,1)$ (see [5]).

The eigenvalues $\lambda_{k}(t)$ of the boundary-value problem (1) and (3),
called as Bloch eigenvalues, are $-(\rho_{k}(t))^{n},$ where $\rho_{k}(t)$ are
the zeroes of the characteristic determinant
\begin{equation}
\Delta(\rho,t):=\det(U_{\nu}(y_{j})_{j,\nu=1}^{n}.
\end{equation}
Here $y_{1}(x,\rho),y_{2}(x,\rho),\ldots,y_{n}(x,\rho)$ are $n$ linearly
independent solutions of $\ $the equation $l(y)=-\rho^{n}y$ satisfying
\begin{equation}
y_{j}^{(\nu-1)}(x,\rho)=\rho^{\nu-1}e^{\rho\omega_{j}x}(\omega_{j}^{\nu
-1}+O\left(  \rho^{-1}\right)  )
\end{equation}
for $\nu=1,2,\cdots,n$, where the numbers $\omega_{1},\omega_{2},\ldots
,\omega_{n}$ denote the different $n$th roots of $-1$.

Besides we use the following result of [12] called here as uniform asymptotic
formulas theorem that were obtained by following the arguments of [6] of the
proof of the asymptotic formulas for the eigenvalues and eigenfunctions and
taking into consideration the uniformity with respect to $t.$

\textbf{Uniform asymptotic formulas theorem}\textit{. If }$n=2\mu+1$\textit{
then for arbitrary }$t\in\mathbb{C}$\textit{, if }$n=2\mu$\textit{ then for
}$t\neq0,\pi,$\textit{ the large eigenvalues of the operator }$L_{t}$\textit{
consist of the sequence }$\{\lambda_{k}(t):\left\vert k\right\vert
>N\},$\textit{ where }$N$ \textit{is a large number}, \textit{satisfying}%
\begin{equation}
\lambda_{k}(t)=\left(  i(2k\pi+t)\right)  ^{n}+O\left(  k^{n-2}\right)
\end{equation}
as $\left\vert k\right\vert \rightarrow\infty.$ \textit{For any fixed }%
$h$\textit{ (}$h\in(0,1)$\textit{) there exists an integer }$N(h)$\textit{
such that the eigenvalue }$\lambda_{k}(t)$\textit{ is simple for all
}$|k|>N(h)$\textit{ if either }$n=2\mu+1$ \textit{and }$t\in$\textit{ }%
$Q$\textit{ or }$n=2\mu$\textit{ and }$t\in Q_{h},$\textit{ where }%
$Q=U_{1}\left(  [0,2\pi]\right)  $ \textit{and} $U_{\delta}\left(
[0,2\pi]\right)  $ \textit{denotes the set of all }$t\in C$\textit{ whose
distance from }$[0,2\pi]$\textit{ is less than} $\delta,$\textit{ }%
$Q_{h}=\{t\in Q:$ $|t-\pi k|\geq h,$ $\forall k=0,1,2\}.$ \textit{Formula (6)
is uniform with respect to }$t$\textit{ in }$Q$\textit{ for }$n=2\mu
+1$\textit{ and in }$Q_{h}$\textit{ for }$n=2\mu.$

\textit{The normalized eigenfunction }$\Psi_{k,t}(x)$\textit{ corresponding to
the eigenvalue }$\lambda_{k}(t)$\textit{\ satisfies }%
\begin{equation}
\Psi_{k,t}(x)=\left(  \parallel e^{itx}\parallel\right)  ^{-1}e^{i(2k\pi
+t)x}+O(k^{-1})
\end{equation}
as $\left\vert k\right\vert \rightarrow\infty.$ \textit{For }$n=2\mu
+1$\textit{ formula (7) is uniform with respect to }$t$ in $Q$\textit{ and
}$x$ \textit{in} $[0,1].$ \textit{For }$n=2\mu$\textit{ it is uniform with
respect to }$t$ \textit{in} $Q_{h}$\textit{ and }$x$ \textit{in} $[0,1].$

Note that a formula \ $f(k,t)=O(g(k))$ as $\left\vert k\right\vert
\rightarrow\infty$ is said to be uniform with respect to $t$ in a set $I$ if
there exist positive constants $N$ and $c$ such that $\mid f(k,t)\mid<c\mid
g(k)\mid$ for all $t\in I$ and $\mid k\mid\geq N.$ Similarly, a formula
\ $f(k,t,x)=O(g(k))$ as $\left\vert k\right\vert \rightarrow\infty$ is said to
be uniform with respect to $t$ in a set $I$ and $x$ in a set $X$ if there
exist positive constants $N$ and $c$ such that $\mid f(k,t,x)\mid<c\mid
g(k)\mid$ for all $t\in I,$ $x\in X$ and $\mid k\mid\geq N.$

From (3)-(5) one can readily see that $\Delta(\lambda,t),$ where
$\lambda=-\rho^{n},$ has the form%
\begin{equation}
\Delta(\lambda,t)=e^{int}+f_{1}(\lambda)e^{i(n-1)t}+f_{2}(\lambda
)e^{i(n-2)t}+...+f_{n}(\lambda),
\end{equation}
that is, $\Delta(\lambda,t)$ is a polynomial of $e^{it}$\ with entire
coefficients $f_{1}(\lambda),$ $f_{2}(\lambda),....$ Therefore the multiple
eigenvalues of the operators $L_{t}$ are the zeros of the resultant
$R(\lambda)$ of the polynomials $\Delta(\lambda,t)$ \ and $\frac{\partial
}{\partial\lambda}\Delta(\lambda,t).$ Since $R(\lambda)$ is an entire
function, we have the following obvious statements formulated in the following remark.

\begin{remark}
For any bounden region $D$ of the complex plane the set of all multiple
eigenvalues of the operators $L_{t}$ for all $t\in\mathbb{C}$ , that is, the
set of all multiple Bloch eigenvalues lying in $D$ is finite. It with uniform
asymptotic formulas theorem implies that if $n=2\mu+1,$ then the set of all
multiple Bloch eigenvalues is a finite set $\left\{  a_{k}%
:k=1,2,...,j\right\}  .$ In the case $n=2\mu$ the set of all multiple Bloch
eigenvalues is either finite set or a sequence $\left\{  a_{k}:k\in
\mathbb{N}\right\}  $ satisfying $a_{k}\rightarrow\infty$ as $k\rightarrow
\infty.$ For each $a_{k}$ there exist at most $n$ values of $t\in Q$
satisfying $\Delta(a_{k},t)=0,$ since $\Delta(\lambda,t)$ is a polynomial of
$e^{it}$\ of order $n$ (see (8)) and it follows from (3) that the quasimomenta
$t$ and $t+2\pi n$ for $n\in\mathbb{Z}$ are the same. Therefore the set
$A:=\cup_{k}A_{k},$ where $A_{k}=\{t\in Q:\Delta(a_{k},t)=0\},$ is either
finite or countable set. Namely, if $n=2\mu+1$ then $A$ is finite. If
$n=2\mu,$ then it follows from the uniform asymptotic formulas theorem that
the possible accumulation points of the set $A\cap Q$ are $0$ and $\pi.$
Denote the set $(A\cap Q)\cup\{0,\pi\}$ by $A^{\prime}.$ If $t\in\left(
Q\backslash A^{\prime}\right)  $ then all eigenvalues of $L_{t}$ are simple
and the system of eigenfunctions of $L_{t}$ forms a Riesz basis.
\end{remark}

Replacing the $i$th row of the determinant $\Delta(\lambda_{k}(t),t)$ (see
(4)) by the row vector

$(y_{1}(x,\lambda_{k}(t)),y_{2}(x,\lambda_{k}(t)),\ldots,y_{n}(x,\lambda
_{k}(t)))$ we get the eigenfunctions $\Psi_{k,i,t}(x)$ for $i=1,2,...,n$
corresponding to the eigenvalue $\lambda_{k}(t).$ It is clear that if $t\notin
A$ then at least one of these eigenfunctions is nonzero, since the rank of the
matrix in (4) is $n-1.$ For $t\notin A$ we denote by $\Psi_{k,t}$ the
normalized eigenfunctions of $L_{t}$ corresponding to the eigenvalue
$\lambda_{k}(t)$ and by $X_{k,t}$ the eigenfunction of the adjoint operator
$L_{t}^{\ast}$ which is orthogonal to all eigenfunctions of $L_{t}$ except
$\Psi_{k,t}$ and $(\Psi_{k,t},X_{k,t})=1.$ It is clear that%
\begin{equation}
X_{k,t}=\left(  \alpha_{k}(t)\right)  ^{-1}\Psi_{k,t}^{\ast},
\end{equation}
where $\Psi_{k,t}^{\ast}$ is the normalized eigenfunctions of $L_{t}^{\ast}$
corresponding to the eigenvalue $\overline{\lambda_{k}(t)}$ and $\alpha
_{k}(t)=(\Psi_{k,t}^{\ast},\Psi_{k,t}).$ Thus $\{X_{k,t}:k\in\mathbb{Z}\}$ is
a biorthogonal system of $\{\Psi_{k,t}:k\in\mathbb{Z}\}.$

\begin{remark}
In Remark 5 we show that the eigenvalues of $L_{t}$ can be numbered (counting
the multiplicity) by elements of the set $\mathbb{Z}$ such that (6) holds and
for each $k$ the function $\lambda_{k}$ is continuous in $[0,2\pi)$ and is
analytic in some neighborhood of $t\in\lbrack0,2\pi)\backslash A(k),$ where
$A(k)$ is a finite subset of $[0,2\pi)\cap A$ such that if $t\notin A(k),$
then $\lambda_{k}(t)$ is a simple eigenvalue. Thus $\sigma(L_{t})=\left\{
\lambda_{k}(t):k\in\mathbb{Z}\right\}  $ and $\sigma(L)$ is the union of the
continuous curves $\Gamma_{k}:=\left\{  \lambda_{k}(t):k\in\lbrack
0,2\pi)\right\}  $ for $k\in\mathbb{Z}.$\ From the above arguments it follows
that $\Psi_{k,t}^{\ast},$ $\Psi_{k,t}$ and $\alpha_{k}(t)$ continuously depend
on $t$ in $[0,2\pi)\backslash A(k).$ Moreover if $t\notin A(k),$ then
$\alpha_{k}(t)\neq0$ since the system $\{\Psi_{k,t}:k\in\mathbb{Z}\}$ is
complete. Therefore $\frac{1}{\alpha_{k}}$ is also continuous in
$[0,2\pi)\backslash A(k).$ Moreover, $\Psi_{k,\cdot}^{\ast},\Psi_{k,\cdot}$
and $\frac{1}{\alpha_{k}}$ are the periodic functions of period $2\pi$ with
respect to the quasimomentum $t.$
\end{remark}

\textbf{(b)} \textbf{Spectral singularities and spectrality of} $L.$ If
$\lambda_{k}(t)$ is a simple eigenvalue of $L_{t},$ then the spectral
projection $e(\lambda_{k}(t))$ defined by contour integration of the resolvent
of $L_{t}(q)$ over the closed contour containing only the eigenvalue
$\lambda_{k}(t)$ has the form%
\begin{equation}
e(\lambda_{k}(t))f=\left(  \overline{\alpha_{k}(t)}\right)  ^{-1}(f,\Psi
_{k,t}^{\ast})\Psi_{k,t},\text{ }\left\Vert e(\lambda_{k}(t))\right\Vert
=\left\vert \left(  \alpha_{k}(t)\right)  ^{-1}\right\vert
\end{equation}
(see [6]). Moreover, in [10] it was proved that if the curve $\gamma
\subset\Gamma_{k}$ does not contain multiple eigenvalues and the common roots
of
\begin{equation}
\Delta(\lambda,t)=0\text{ }\And\frac{\partial\Delta(\lambda,t)}{\partial t}=0,
\end{equation}
then for the projection $P(\gamma)$ of $L$ corresponding to $\gamma$ we have
\begin{equation}
\left\Vert P(\gamma)\right\Vert =\sup\nolimits_{t\in\delta}\left\vert \left(
\alpha_{k}(t)\right)  ^{-1}\right\vert ,
\end{equation}
where $\delta=\left\{  t\in\lbrack0,2\pi):\lambda_{k}(t)\in\gamma\right\}  $.
Then the following definition was given.

\begin{definition}
We say that $\lambda\in\sigma(L)$ is a spectral singularity of $L$ if for all
$\varepsilon>0$\ there exists a sequence $\{\gamma_{n}\}$ of arcs $\gamma
_{n}\subset\{z\in\sigma(L):\mid z-\lambda\mid<\varepsilon\}$ such that they do
not contain multiple eigenvalues and common roots of (11) and
\begin{equation}
\lim_{n\rightarrow\infty}\parallel P(\gamma_{n})\parallel=\infty.
\end{equation}

\end{definition}

Note that a similar definition for the case $n=2$ was given in [3]. In the
similar way, we defined in [14, 15] the spectral singularity at infinity.

\begin{definition}
We say that the operator $L$ has a spectral singularity at infinity if there
exists a sequence $\{\gamma_{n}\}$ of the arcs of $\sigma(L)$ such that they
do not contain multiple eigenvalues and common roots of (11) and
$d(0,\gamma_{n})\rightarrow\infty$ as $n\rightarrow\infty$ and (13) holds,
where $d(0,\gamma_{n})$ is the distance from the point $(0,0)$ to the arc
$\gamma_{n}.$
\end{definition}

The existence of the spectral singularities does not allow $L$ to be the
spectral operator, since McGarvey [4] proved that $L$ is a spectral operator
if and only if the projections of the operators $L_{t}$ are bounded uniformly
with respect to $t$ in $[0,2\pi)$ and we have the estimations (10) and (12).
Moreover, Gasymov's paper [1] shows that, in general, the operator $L$ for
$n=2$ has infinitely many spectral singularities. Nevertheless, Gesztezy and
Tkachenko [3] proved two versions of a criterion for the Hill operator to be a
spectral operator of scalar type, in sense of Dunford, one analytic and one
geometric. The analytic version was stated in term of the solutions of Hill
equation. The geometric version of the criterion uses algebraic and geometric
\ properties of the spectra of periodic/antiperiodic and Dirichlet boundary
value problems. However, the discussed papers show that the set of potentials
$q$ for which the Hill operator $H(q)$ with potential $q$ is spectral is a
small subset of the set of periodic functions and it is very hard to describe
explicitly the required subset. In papers [14, 15] we found the explicit
conditions on the potential $q$ such that $H(q)$ is an asymptotically spectral
operator. The set of the potentials constructed in [14, 15] is also a small
subset of the set of the periodic functions. Thus the theory of spectral
operators is ineffective for construction of the spectral expansion for the
nonself-adjoint periodic differential operators.

\textbf{(c) Spectral expansion of} $L.$ By Gelfand's Lemma (see [2]) for every
$f\in L_{2}(-\infty,\infty)$ there exists $\Upsilon f$ $\in L_{2}%
([0,1]\times\lbrack0,2\pi])$ such that
\begin{equation}
\text{ }f(x)=\frac{1}{2\pi}\int\nolimits_{0}^{2\pi}f_{t}(x)dt,\text{ }\frac
{1}{2\pi}\int\nolimits_{0}^{2\pi}\int_{0}^{1}\left\vert f_{t}(x)\right\vert
^{2}dxdt=\int_{\mathbb{R}}\left\vert f(x)\right\vert ^{2}dx,
\end{equation}
where%
\begin{equation}
f_{t}(x):=\left(  \Upsilon f\right)  (x,t)=\sum\nolimits_{k=-\infty}^{\infty
}f(x+k)e^{-ikt},\text{ }f_{t}(x+1)=e^{it}f_{t}(x).
\end{equation}
Let $-h\in(0,\frac{1}{2})\backslash A^{\prime}$ and $l$\ be a continuous curve
joining the points $-h$ and $2\pi-h$ and satisfying
\begin{equation}
l\subset Q\backslash A^{\prime}\text{ }\And l\subset Q_{h}\backslash
A^{\prime}\text{ }%
\end{equation}
if $n=2\mu+1$ and $n=2\mu$ respectively. If $f_{t}(x)$ analytically depend on
$t$ in $\overline{D}$ for each $x,$ where $D$ is the domain enclosed by
$l\cup\lbrack-h,2\pi-h]$ and everywhere $\overline{D}$ denotes the closure of
$D,$ then using (14), (15) and Cauchy's theorem we obtain
\begin{equation}
f(x)=\frac{1}{2\pi}\int\nolimits_{0}^{2\pi}f_{t}(x)dt=\frac{1}{2\pi}\int
_{l}f_{t}(x)dt.
\end{equation}
Since the last equality play the crucial role for the spectral expansion
theorem, to avoid the eclipsing the essence of the paper by technical details
we consider the set $E$ of $f\in L_{2}(-\infty,\infty)$ satisfying the
following conditions (however, the obtained results can be easily curry out
for large set).

\begin{condition}
$(i)$ For each $f\in E$ there exists $\beta>0$ such that the Gelfand transform
$\Upsilon f$ and its derivative with respect to $t$ are continuous functions
of two variables $\left(  x,t\right)  $ on $[0,1]\times U_{\beta}\left(
[0,2\pi]\right)  .$ $(ii)$ Fourier transform $g(\lambda)$ of $f\in E$
approaches zero as $\lambda\rightarrow\infty.$
\end{condition}

By (16), for each $t\in l$ the system $\{\Psi_{k,t}:k\in\mathbb{Z}\}$ form a
Reisz basis of $L_{2}[0,1]$. Therefore
\begin{equation}
f_{t}(x)=\sum_{k\in\mathbb{Z}}a_{k}(t)\Psi_{k,t}(x),
\end{equation}
where
\begin{equation}
a_{k}(t)=\int_{0}^{1}f_{t}(x)\overline{X_{k,t}(x)}dx=\int_{\mathbb{R}%
}f(x)\overline{X_{k,t}(x)}dx,
\end{equation}
$X_{k,t}$ is defined in (9), and $\Psi_{k,t}(x+1)=e^{it}\Psi_{k,t}(x),$
$\Psi_{k,t}^{\ast}(x+1)=e^{i\overline{t}}\Psi_{k,t}^{\ast}(x).$ Using (18) in
(17), we get
\begin{equation}
f(x)=\frac{1}{2\pi}\int\nolimits_{l}f_{t}(x)dt=\frac{1}{2\pi}\int
\nolimits_{l}\sum\nolimits_{k\in\mathbb{Z}}a_{k}(t)\Psi_{k,t}(x)dt.
\end{equation}

Repeating the proof of the\ term by term integration of the series in (20)
which was given in [11] (for this one can look also the Appendix A of [13]) we
obtain that for $f\in E$ and for any piecewise smooth curve $\gamma$ that is
compact subset of $Q\backslash A$ and $Q_{h}\backslash A$ for the odd and even
cases respectively, the following holds
\begin{equation}
\int\nolimits_{\gamma}\sum\nolimits_{k\in\mathbb{Z}}a_{k}(t)\Psi
_{k,t}(x)dt=\sum\nolimits_{k\in\mathbb{Z}}\int\nolimits_{\gamma}a_{k}%
(t)\Psi_{k,t}(x)dt.
\end{equation}
Hence by (20) we have
\begin{equation}
f(x)=\frac{1}{2\pi}\sum_{k\in\mathbb{Z}}\int\nolimits_{l}a_{k}(t)\Psi
_{k,t}(x)dt,
\end{equation}
where the series converges in the norm of $L_{2}(a,b)$ for every
$a,b\in\mathbb{R}$.

To get the spectral expansion in the term of $t$ from (22) we need to consider
the integrability of $a_{k}(t)\Psi_{k,t}(x)$ over $[0,2\pi)$ for almost all
$x$ and then replace the integrals over $l$ by the integrals over $[0,2\pi)$.
In Theorem 9 using the equality
\begin{equation}
a_{k}(t)\Psi_{k,t}(x)=\frac{1}{\overline{\alpha_{k}(t)}}\left(  \int_{0}%
^{1}f_{t}(x)\overline{\Psi_{k,t}^{\ast}(x)}dx\right)  \Psi_{k,t}(x)
\end{equation}
(see (9) and (19)) we prove that the existence of the integral of
$a_{k}(t)\Psi_{k,t}(x)$ with respect to $t$ over $[0,2\pi)$ can be reduced to
the investigation of the integrability of $\frac{1}{\alpha_{k}}.$ Therefore we
introduce the following notions which is independent of $f$ and depend only on
the spectral property of $L$.

\begin{definition}
A number $\lambda_{0}\in\sigma(L)$ is said to be an essential spectral
singularity (ESS) of $L$ if there exist $t_{0}\in\lbrack0,2\pi)$ and
$k\in\mathbb{Z}$ such that $\lambda_{0}=\lambda_{k}(t_{0})$ and $\frac
{1}{\alpha_{k}}$ is not integrable over $(t_{0}-\delta,t_{0}+\delta)$ for all
$\delta>0.$ Then $t_{0}$ is called a singular quasimomentum (SQ).
\end{definition}

It is clear that $\lambda_{0}=\lambda_{k}(t_{0})$ is ESS if and only is there
exists sequence of closed intervals $I(s)$ approaching $t_{0\text{ }}$such
that $\lambda_{k}(t)$ for $t\in I(s)$ is a simple eigenvalue and
\begin{equation}
\lim_{s\rightarrow\infty}\int\nolimits_{I(s)}\left\vert \alpha_{k}%
(t)\right\vert ^{-1}dt=\infty.
\end{equation}

It the similar way we define ESS at infinity.

\begin{definition}
We say that the operator $L$ has ESS at infinity if there exists sequence of
integers $k_{s}$ and sequence of closed subsets $I(s)$ of $[0,2\pi)\backslash
A(k_{s}),$where $A(k)$ is defined in Remark 2, such that
\begin{equation}
\lim_{s\rightarrow\infty}\int\nolimits_{I(s)}\left\vert \alpha_{k_{s}%
}(t)\right\vert ^{-1}dt=\infty.
\end{equation}

\end{definition}

It readily follows from the Definitions 2 and 4 that if $L$ has ESS at
infinity then it has spectral singularity at infinity too. However, the
existence of the spectral singularity at infinity does not imply that it has
ESS at infinity.

Note that it follows from (12) and Definitions 1 and 2 that the boundlessness
of $\frac{1}{\alpha_{k}}$ is the characterization of the spectral
singularities and the considerations of the spectral singularities play only
the crucial rule for the investigations of the spectrality of $L$. On the
other hand, the papers [1, 3, 14-16] show that the periodic differential
operators, in general, is not a spectral operator. Therefore to construct the
spectral expansion for the operator $L$ in the general case we need to
introduce the new concepts ESS and SQ connected with the nonintegrability of
$\frac{1}{\alpha_{k}}$, since $\frac{1}{\alpha_{k}}$ may have an integrable
boundlessness. In fact, to construct the spectral expansion we need to
consider the existence of the integral with respect to $t$ over $[0,2\pi)$ of
the function in (23) for almost all $x$.

Note also that by Remark 2 the function in (23) and $\frac{1}{\alpha_{k}}$ are
defined and continuous in $[0,2\pi)\backslash A(k),$ where $A(k)$ is a finite
set. Hence the Lebesque integral of these functions over $I\subset
\lbrack0,2\pi)$ and $I\backslash A(k)$ are the same. Now (12) and definitions
1 and 3 imply the following.

\begin{proposition}
The set of ESS is the subset of the set of spectral singularities and the set
of spectral singularities is the subset of the set of multiple eigenvalues of
$L_{t}$ for $t\in\lbrack0,2\pi)$.
\end{proposition}

In this paper we construct the spectral expansion for $f\in E$ by using the
concepts SQ, ESS and ESS at infinity. First (in Section 2) we consider the
case $n=2\mu+1$ which is simpler than $n=2\mu$. Then, in Section 3, we
investigate the complicated case $n=2\mu.$ Moreover, we prove that the
spectral decomposition has the elegant form (as in the self-adjoint case [2])
\begin{equation}
f(x)=\frac{1}{2\pi}\sum_{k\in\mathbb{Z}}\int\nolimits_{[0,2\pi)}a_{k}%
(t)\Psi_{k,t}(x)dt
\end{equation}
if and only if $L$ has no ESS and ESS at infinity. Note that this statement is
new for $n=2$ too and is the proof of the necessity of the introducing the
concepts SQ, ESS and ESS at infinity. Note also that, in general, the SQ, ESS
and ESS at infinity exist and hence (26) does not hold. That is why, it is
necessary to find a new form for the spectral expansion. We prove that it is
necessary to use the parenthesis in the spectral expansion series for the
general case. In this paper we try to use all and necessary factors that
effect to the spectral expansion of $L$ (see the conclusion at the end of the paper).

\section{Spectral Expansion for Odd Order}

In this section we consider the simply case $n=2\mu+1$ so that it helps to
read the complicated case $n=2\mu$. The main reason of the simplicity of the
odd case is that in the case the number of elements of $A\cap Q$ and the
number of multiple Bloch eigenvalues are finite (see Remark 1). Denote the
points of $A\cap\lbrack0,2\pi)$ by $t_{1}<t_{2}<...<t_{s}$ and introduce the
set
\begin{equation}
\mathbb{T}(v,j):=\left\{  k\in\mathbb{Z}:\lambda_{k}(t_{j})=\Lambda_{v}%
(t_{j})\right\}  ,
\end{equation}
where $\Lambda_{1}(t_{j}),$ $\Lambda_{2}(t_{j}),...,\Lambda_{s_{j}}(t_{j})$
are the different multiple eigenvalues of the operator $L_{t_{j}}.$ The set
$\mathbb{T}(v,j)$ is finite, since the multiplicity of $\Lambda_{v}(t_{j})$ is
finite. Moreover,
\begin{equation}
\mathbb{T}(v,j)\cap\mathbb{T}(m,j)=\emptyset,\text{ }\forall m\neq v
\end{equation}
due to $\Lambda_{v}(t_{j})\neq\Lambda_{m}(t_{j}).$ However, $\mathbb{T}%
(v,j)\cap\mathbb{T}(v,i)$ for $i\neq j$ and $\mathbb{T}(v,j)\cap
\mathbb{T}(m,i)$ for $i\neq j$ and $m\neq v$ must not to be the empty sets.

If $\Lambda_{v}(t_{j})$ is the $p$multiple eigenvalue of $L_{t_{j}},$ then it
is the $p$-multiple root of the equation $\Delta(\lambda,t_{j})=0.$ Therefore
using the implicit function theorem for (8) and then the properties of the
Grean function of $L_{t}$ we obtain the following proposition and theorem
which hold for the both (odd and even) cases.

\begin{proposition}
If $\Lambda_{v}(t_{j})$ is the eigenvalue of $L_{t_{j}}$ of multiplicity $p$
then there exist a disk $U(t_{j},\varepsilon_{j})=\{t\in\mathbb{C}:\left\vert
t-t_{j}\right\vert <\varepsilon_{j}\}$ and $r_{2}>r_{1}>0$ such that the
followings hold.

$(i)$ In $r_{1}$ neighborhood of $\Lambda_{v}(t_{j})$ the operators $L_{t}$
for $t\in U(t_{j},\varepsilon_{j})\backslash\{t_{j}\}$ has $p$ simple
eigenvalues and these eigenvalues are
\begin{equation}
\lambda_{k}(t)\text{ for }k\in\mathbb{T}(v,j).
\end{equation}

$(ii)$ In $r_{2}$ neighborhood of $\Lambda_{v}(t_{j})$ the operators $L_{t}$
for $t=t_{j}$ and $t\in U(t_{j},\varepsilon_{j})\backslash\{t_{j}\}$ have no
other eigenvalues than $\Lambda_{v}(t_{j})$ and (29) respectively.
\end{proposition}

Now using this proposition we prove the following theorem that helps to
replace the integral over $l$ by the integral over $[0,2\pi)$ for $f\in E,$
where $E$ is defined in Condition 1.

\begin{theorem}
If $f\in E,$ then for each $x\in(-\infty,\infty)$ the expression%
\begin{equation}
\sum_{k\in\mathbb{T}(v,j)}a_{k}(t)\Psi_{k,t}(x)
\end{equation}
analytically depend on $t$ in $U(t_{j},\varepsilon_{j})\backslash\{t_{j}\}$,
is integrable over $(t_{j}-\varepsilon,t_{j}+\varepsilon)$ and
\begin{equation}
\sum_{k\in\mathbb{T}(v,j)}\int\nolimits_{\gamma(t_{j},\varepsilon)}%
a_{k}(t)\Psi_{k,t}(x)dt=\int\limits_{(t_{j}-\varepsilon,t_{j}+\varepsilon
)}\sum_{k\in\mathbb{T}(v,j)}a_{k}(t)\Psi_{k,t}(x)dt,\text{ }\forall
j=1,2,...,s,
\end{equation}
where $\ $ $\varepsilon<\min\left\{  \beta,\varepsilon_{1},\varepsilon
_{2},...,\varepsilon_{s}\right\}  ,$ $\gamma(t_{j},\varepsilon)=\{t\in
\mathbb{C}:\left\vert t-t_{j}\right\vert =\varepsilon,\operatorname{Im}%
t\geq0\},$ $\varepsilon_{j}$ and $a_{k}(t)\Psi_{k,t}(x)$ are defined in
Proposition 2 and (23) respectively and $\beta$ is a positive constant defined
in Condition 1 for $f.$ If $\lambda_{k}(t_{0})$ is a simple eigenvalue, then
$a_{k}(t)\Psi_{k,t}(x)$ analytically depend on $t$ in some neighborhood of
$t_{0}.$
\end{theorem}

\begin{proof}
Let $C$ be the circle with center $\Lambda_{v}(t_{j})$ and radius $r,$ where
$r_{1}<r<$ $r_{2}$ and the numbers $r_{1}$ and $r_{2}$ are defined in
Proposition 2. Consider the total projection
\begin{equation}
T(x,t):=\int_{C}A(x,\lambda,t)d\lambda,
\end{equation}
where
\begin{equation}
A(x,\lambda,t):=\int\nolimits_{0}^{1}G(x,\xi,\lambda,t)f_{t}(\xi)d\xi,
\end{equation}
$f_{t}$ is defined in (15) and $G(x,\xi,\lambda,t)$ is the Green function of
the operator $L_{t}.$ It is well-known that $G(x,\xi,\lambda,t)$ is defined by
formula (see [6] pages 36 and 37)%
\begin{equation}
G(x,\xi,\lambda,t)=\frac{H(x,\xi,\lambda,t)}{\Delta(\lambda,t)},
\end{equation}
where $H(x,\xi,\lambda,t)$ is the $(n+1)\times(n+1)$ determinant defined as
follows
\begin{equation}
H(x,\xi,\lambda,t)=\left\vert
\begin{array}
[c]{ccccc}%
y_{1}(x) & y_{2}(x) & \cdot\cdot\cdot & y_{n}(x) & g(x,\xi)\\
U_{1}(y_{1}) & U_{1}(y_{2}) & \cdot\cdot\cdot & U_{1}(y_{n}) & U_{1}(g)\\
U_{2}(y_{1}) & U_{2}(y_{2}) & \cdot\cdot\cdot & U_{2}(y_{n}) & U_{2}(g)\\
\cdot\cdot\cdot & \cdot\cdot\cdot & \cdot\cdot\cdot & \cdot\cdot\cdot &
\cdot\cdot\cdot\\
U_{n}(y_{1}) & U_{n}(y_{2}) & \cdot\cdot\cdot & U_{n}(y_{n}) & U_{n}(g)
\end{array}
\right\vert ,
\end{equation}%
\begin{equation}
g(x,\xi)=\pm\frac{1}{2W(\xi)}%
\begin{vmatrix}
y_{1}(x) & y_{2}(x) & \cdots & y_{n}(x)\\
y_{1}^{(n-2)}(\xi) & y_{2}^{(n-2)}(\xi) & \cdots & y_{n}^{(n-2)}(\xi)\\
y_{1}^{(n-3)}(\xi) & y_{2}^{(n-3)}(\xi) & \cdots & y_{n}^{(n-3)}(\xi)\\
. & . & \cdots & .\\
y_{1}(\xi) & y_{2}(\xi) & ... & y_{n}(\xi)
\end{vmatrix}
\end{equation}
and $W(\xi)$ is the Wronskian of the solutions $y_{1},y_{2},\ldots,y_{n}.$ In
(36) the positive sign being taken if $x>\xi$ and the negative sign if
$x<\xi.$

By Proposition 2 the circle $C$ encloses only the eigenvalues (29) and lies in
the resolvent sets of $L_{t}$ for $t\in U(t_{j},\varepsilon_{j}).$ By (8),
$\Delta(\lambda,t)$ is continuous in the compact $C\times\overline
{U(t_{j},\varepsilon)}$ and hence there exists a positive constant $c$ such
that
\begin{equation}
\left\vert \Delta(\lambda,t)\right\vert \geq c,\text{ }\forall(\lambda,t)\in
C\times U(t_{j},\varepsilon_{j}).
\end{equation}
Therefore using Condition 1$(i)$ and Lebesgue dominated convergence theorem
from (32)-(37) we obtain the following results.

$(i)$ There exists $M_{v,j}$ such that $\left\vert A(x,\lambda,t)\right\vert
\leq M_{v,j}$ for all $(x,\lambda,t)\in\lbrack0,1]\times C\times
U(t_{j},\varepsilon_{j})$.

$(ii)$ $A(x,\lambda,t)$ analytically depend on $t$ in $U(t_{j},\varepsilon
_{j})$ for $(x,\lambda)\in\lbrack0,1]\times C.$

$(iii)$ $A(x,\lambda,t)$ and $\frac{\partial A}{\partial t}(x,\lambda,t)$
continuously depend on $(\lambda,t)$ in $C\times U(t_{j},\varepsilon_{j})$ for
$x\in\lbrack0,1].$

$(iiii)$ $T(x,t)$ analytically depend on $t$ in $U(t_{j},\varepsilon_{j})$ for
$x\in\lbrack0,1]$ and $\left\vert T(x,t)\right\vert \leq4\pi M_{v,j}r$ for all
$(x,t)\in\lbrack0,1]\times U(t_{j},\varepsilon_{j})$.

The proof of $(i)-(iii)$ follows from (33)-(37) and the proof of $(iiii)$
follows from (32) and $(i)-(iii).$ Since inside of $C$ the operator $L_{t}$
for $t\in U(t_{j},\varepsilon_{j})\backslash\{t_{j}\}$ has $p$ simple
eigenvalues (29) we have%
\begin{equation}
T(x,t)=\sum_{k\in\mathbb{T}(v,j)}a_{k}(t)\Psi_{k,t}(x),\text{ }\forall t\in
U(t_{j},\varepsilon_{j})\backslash\{t_{j}\}.
\end{equation}
It with $(iiii)$ implies that (30) analytically depend on $t$ in
$U(t_{j},\varepsilon_{j})\backslash\{t_{j}\}$ and is integrable over
$(t_{j}-\varepsilon,t_{j}+\varepsilon)$ for all $x.$ On the other hand, again
by $(iiii),$ we have
\begin{equation}
\int\nolimits_{\gamma(t_{j},\varepsilon)}T(x,t)dt=\int\nolimits_{(t_{j}%
-\varepsilon,t_{j}+\varepsilon)}T(x,t)dt.
\end{equation}
Moreover, by Proposition 2, for $t\in\gamma(t_{j},\varepsilon)$ the
eigenvalues (29) are simple and hence $a_{k}(t)\Psi_{k,t}(x)$ continuously
depend on $t$ in $\gamma(t_{j},\varepsilon)$ for each $k\in\mathbb{T}(v,j)$
and $x\in\lbrack0,1].$ Therefore
\begin{equation}
\int\limits_{\gamma(t_{j},\varepsilon)}\sum_{k\in\mathbb{T}(v,j)}a_{k}%
(t)\Psi_{k,t}(x)dt=\sum_{k\in\mathbb{T}(v,j)}\int\limits_{\gamma
(t_{j},\varepsilon)}a_{k}(t)\Psi_{k,t}(x)dt.
\end{equation}
Thus (31) follows from (38)-(40). It remains to note that $T(x,t)=a_{k}%
(t)\Psi_{k,t}(x)$ if $\lambda_{k}(t)$ is a simple eigenvalue
\end{proof}

Let $h$ and $\varepsilon$ be positive numbers such that $-h\notin A,$
$t_{s}<2\pi-h<2\pi$ and
\begin{equation}
2\varepsilon<\min_{j=1,2,...,s-1}\{t_{1}+h,t_{j+1}-t_{j},2\pi-h-t_{s}\},
\end{equation}
that is, $\varepsilon$ is less than half of the length of the intervals
$(-h,t_{1}),$ $(t_{1},t_{2}),...,$ $(t_{s-1},t_{s})$ and $(t_{s},2\pi-h)$ (see
(27) for $t_{j}$). Besides, we assume that $\varepsilon$ satisfies the
inequality in Theorem 1. Let $l(\varepsilon)$ be a curve joining the points
$-h$ and $2\pi-h,$ passing over the points $t_{1}<t_{2}<...<t_{s}$ and
consisting of the intervals
\begin{equation}
\lbrack-h,t_{1}-\varepsilon),(t_{1}+\varepsilon,t_{2}-\varepsilon
),...,(t_{s-1}+\varepsilon,t_{s}-\varepsilon),(t_{s}+\varepsilon,2\pi-h]
\end{equation}
and semicircles $\gamma(t_{j},\varepsilon)$ for $j=1,2,...,s$ defined in
Theorem 1. Thus
\begin{equation}
l(\varepsilon)=\left(
{\textstyle\bigcup\nolimits_{j=1}^{s}}
\gamma(t_{j},\varepsilon)\right)  \cup\left(
{\textstyle\bigcup\nolimits_{j=1}^{s-1}}
(t_{j}+\varepsilon,t_{j+1}-\varepsilon)\right)  \cup\lbrack-h,t_{1}%
-\varepsilon)\cup(t_{s}+\varepsilon,2\pi-h].
\end{equation}

The following statements signify the reason of the simplicity of the odd case.

\begin{proposition}
Let $n=2\mu+1$, $f\in E$ and $x\in(-\infty,\infty).$ Then

$(a)$ For each $k\in\mathbb{Z}$, $a_{k}(t)\Psi_{k,t}(x)$ is integrable on the
sets in (42) and (43).

$(b)$ The curve $l(\varepsilon)$ can be chosen so that it and the open domain
$D(\varepsilon)$ enclosed by the curve $l(\varepsilon)\cup\lbrack-h,2\pi-h]$
does not contain the point of the set $A$. Moreover, if $\lambda_{k}(t_{j})$
is a simple eigenvalue for $j=1,2,...,s,$ then
\begin{equation}
\int_{l(\varepsilon)}a_{k}(t)\Psi_{k,t}(x)dt=\int_{[0,2\pi)}a_{k}(t)\Psi
_{k,t}(x)dt.
\end{equation}

$(c)$ The operator $L$ has at most finite number SQ and ESS and has no
spectral singularity and ESS at infinity.
\end{proposition}

\begin{proof}
The proof of $(a)$ follows from Theorem 1.

$(b)$ Since $A\cap Q$ is a finite set, one can choose $l(\varepsilon)$ so that
it and $D(\varepsilon)$ does not contain the point of the set $A$. If in
addition $\lambda_{k}(t_{j})$ is a simple eigenvalue for $j=1,2,...,s,$ then
$\lambda_{k}(t)$ is a simple eigenvalue for all $t\in\overline{D(\varepsilon
)}.$ Therefore, $a_{k}(t)\Psi_{k,t}(x)$ analytically depend on $t$ in
$\overline{D(\varepsilon)}$ for all $x.$ Thus using it and the equality
\ $a_{k}(t+2\pi)\Psi_{k,t+2\pi}(x)=a_{k}(t)\Psi_{k,t}(x)$ we obtain (44).

$(c)$ Using (7) and taking into account that under condition (2) the
eigenfunction $\Psi_{k,t}^{\ast}$ of the adjoint operator $L_{t}^{\ast}$ for
$t\in\lbrack0,2\pi)$ also satisfies (7) we obtain \
\begin{equation}
\alpha_{k}(t)=(\Psi_{k,t}^{\ast},\Psi_{k,t})=1+O(k^{-1})
\end{equation}
as $k\rightarrow\infty.$ It is uniform with respect to $t$ in $[0,2\pi).$
Therefore the proof of $(c)$ follows from (12) and Definitions 1-4
\end{proof}

Now we try to replace the semicircle $\gamma(t_{j},\varepsilon)$ by the
interval $(t_{j}-\varepsilon,t_{j}+\varepsilon),$ when $\lambda_{k}(t_{j})$ is
not a simple eigenvalue, that is, we prove the equality
\begin{equation}
\int\nolimits_{\gamma(t_{j},\varepsilon)}a_{k}(t)\Psi_{k,t}(x)dt=\int
\nolimits_{(t_{j}-\varepsilon,t_{j}+\varepsilon)}a_{k}(t)\Psi_{k,t}(x)dt
\end{equation}
for some value of $k$. More precisely, we determine whether (46) and (44) hold
or not in order to replace in (22) the curve $l(\varepsilon)$ by $[-h,2\pi-h]$
and get a spectral expansion for $L.$

\begin{remark}
Note that by Theorem 1, for each $x$ the expression in (30) is integrable on
$(t_{j}-\varepsilon,t_{j}+\varepsilon)$. However, in general, the summands of
(30) for same values of $k$ may became nonintegrable. We say that the set%
\begin{equation}
\left\{  a_{k}(t)\Psi_{k,t}(x):k\in\mathbb{T}(v,j)\right\}
\end{equation}
is a bundle corresponding to the multiple eigenvalue $\Lambda_{v}(t_{j}),$
where $\mathbb{T}(v,j)$ is defined in (27). If $\Lambda_{v}(t_{j})$ is not
ESS, then it follows \ from Definition 3 and Theorem 9 that all elements of
the bundle (47) are integrable on $(t_{j}-\varepsilon,t_{j}+\varepsilon)$ for
all $x$. Therefore by (31) we have
\begin{equation}
\sum_{k\in\mathbb{T}(v,j)}\int\nolimits_{\gamma(t_{j},\varepsilon)}%
a_{k}(t)\Psi_{k,t}(x)dt=\sum_{k\in\mathbb{T}(v,j)}\int\nolimits_{(t_{j}%
-\varepsilon,t_{j}+\varepsilon)}a_{k}(t)\Psi_{k,t}(x)dt,
\end{equation}
since $\mathbb{T}(v,j)$ consist of the finite number of indices.
\end{remark}

Now using (22), (43) and (48) we obtain

\begin{theorem}
If $n=2\mu-1$ and $L$ has no ESS or equivalently has no SQ, then for $f\in E$
spectral expansion (26) holds, where the series converges in the norm of
$L_{2}(a,b)$ for every $a,b\in\mathbb{R}.$
\end{theorem}

Now we consider the case when $L$ has the ESS. By Proposition 1, the set of
the ESS is the subset of the set of multiple eigenvalues
\begin{equation}
\left\{  \Lambda_{v}(t_{j}):j=1,2,...,s;\text{ }v=1,2,...,s_{j}\right\}
\end{equation}
and (49) is a finite set. Therefore, for the simplicity of the notation and
without loss of generality, we re-numarete the elements of (49) so that
\begin{equation}
\left\{  t_{1},t_{2},...,t_{m}\right\}  \text{ and }\left\{  \Lambda_{v}%
(t_{j}):j=1,2,...,m;\text{ }v=1,2,...,m_{j}\right\}  ,
\end{equation}
where $m\leq s$ and $m_{j}\leq s_{j},$ denote the set of SQ and ESS respectively.

If $\Lambda_{v}(t_{j})$ is an ESS, then for some values of $k\in
\mathbb{T}(v,j)$ the function $a_{k}(t)\Psi_{k,t}(x)$ for some $f$ and $x$ is
nonintegrable on $(t_{j}-\varepsilon,t_{j}+\varepsilon),$ while some of
elements of the bundle (47) may be integrable and the total sum of elements of
(47) is integrable due to the cancellations of the nonintegrable terms. At
least two element of the bundle must be nonintegrable in order to do the
cancellations. In fact, we may huddle together only the nonintegrable elements
of (47). To do this handling we prove the following proposition.

\begin{proposition}
Let $\Lambda_{v}(t_{j})$ be an ESS, and $\mathbb{S}(v,j)$ be the subset of
$\mathbb{T}(v,j)$ such that for $k\in\mathbb{S}(v,j)$ the function $\frac
{1}{\alpha_{k}}$ is nonintegrable in $(t_{j}-\varepsilon,t_{j}+\varepsilon).$
Then for $f\in E$ the expression
\begin{equation}
S_{v,j}(x,t):=\sum\nolimits_{k\in\mathbb{S}(v,j)}a_{k}(t)\Psi_{k,t}(x)
\end{equation}
is integrable over $(t_{j}-\varepsilon,t_{j}+\varepsilon)$,
\begin{equation}
\sum_{k\in\mathbb{T}(v,j)}\int\nolimits_{\gamma(t_{j},\varepsilon)}%
a_{k}(t)\Psi_{k,t}dt=
\end{equation}%
\[
\sum_{k\in\mathbb{T}(v,j)\backslash\mathbb{S}(v,j)}\int\nolimits_{(t_{j}%
-\varepsilon,t_{j}+\varepsilon)}a_{k}(t)\Psi_{k,t}dt+\int\nolimits_{(t_{j}%
-\varepsilon,t_{j}+\varepsilon)}\sum_{k\in\mathbb{S}(v,j)}a_{k}(t)\Psi
_{k,t}dt
\]
and
\begin{equation}
\int\limits_{(t_{j}-\varepsilon,t_{j}+\varepsilon)}\sum_{k\in\mathbb{S}%
(v,j)}a_{k}(t)\Psi_{k,t}(x)dt=\lim_{\delta\rightarrow0}\sum_{k\in
\mathbb{S}(v,j)}\int\limits_{\delta<\left\vert t-t_{j}\right\vert
\leq\varepsilon}a_{k}(t)\Psi_{k,t}(x)dt.
\end{equation}

\end{proposition}

\begin{proof}
By definition of $\mathbb{S}(v,j)$ and Theorem 9, $a_{k}(t)\Psi_{k,t}(x)$ is
integrable on $(t_{j}-\varepsilon,t_{j}+\varepsilon)$ for $k\in\mathbb{T}%
(v,j)\backslash\mathbb{S}(v,j)$ and $x\in(-\infty,\infty).$ Thus taking into
account that (30) is integrable on $(t_{j}-\varepsilon,t_{j}+\varepsilon)$ and
using (31) we obtain that $S_{v,j}(x,t)$ is also integrable on $(t_{j}%
-\varepsilon,t_{j}+\varepsilon)$ and (52) holds. Now (53) follows from the
absolute continuity of the integral
\end{proof}

Introduce the notations
\begin{equation}
\mathbb{S}_{j}=%
{\textstyle\bigcup\nolimits_{v=1}^{m_{j}}}
\mathbb{S}(v,j),\text{ }\mathbb{S}=%
{\textstyle\bigcup\nolimits_{j=1}^{m}}
\mathbb{S}_{j},\text{ }\mathbb{T}_{j}=%
{\textstyle\bigcup\nolimits_{v=1}^{s_{j}}}
\mathbb{T}(v,j),\text{ }\mathbb{T}=%
{\textstyle\bigcup\nolimits_{j=1}^{s}}
\mathbb{T}_{j},
\end{equation}%
\begin{equation}
S_{j}(x,t)=\sum\nolimits_{k\in\mathbb{S}_{j}}a_{k}(t)\Psi_{k,t}(x),\text{
}S(x,t)=\sum\nolimits_{k\in\mathbb{S}}a_{k}(t)\Psi_{k,t}(x).
\end{equation}
Here $\mathbb{S}$\ and $\mathbb{T}$ are finite subsets of $\mathbb{Z}$, since
$\mathbb{S}\subset\mathbb{T}$ and the number of elements of $\mathbb{T}(v,j)$
is equal to the multiplicity of the eigenvalue $\Lambda_{v}(t_{j})$ of the
operator $L_{t_{j}}.$

The definitions of $\mathbb{S}(v,j),$ $\mathbb{S}_{j},$ $\mathbb{T}%
_{j},\ \mathbb{S}$\ and $\mathbb{T}$ immediately imply the following.

\begin{proposition}
$(a)$The relations $k\in\mathbb{S}$\ and $k\in\mathbb{T}\backslash\mathbb{S}$
hold respectively if and only if $\frac{1}{\alpha_{k}}$ is nonintegrable and
integrable over $[0,2\pi).$

$(b)$ $\mathbb{S}_{j}$ and $\mathbb{S}$ are the set of all $k$ for which
$\frac{1}{\alpha_{k}}$ is nonintegrable over $(t_{j}-\varepsilon
,t_{j}+\varepsilon)$ and $[0,2\pi)$ respectively.

$(c)$ $\mathbb{T}_{j}$ and $\mathbb{T}$ are the set of all $k$ for which
$\lambda_{k}(t)$ is a multiple eigenvalue for $t=t_{j}$ and for some
$t\in\left\{  t_{1},t_{2},...,t_{s}\right\}  $ respectively.
\end{proposition}

To replace the circles $\gamma(t_{j},\varepsilon)$ by the intervals
$(t_{j}-\varepsilon,t_{j}+\varepsilon)$ for $j=1,2,...,s$ and hence to get
integrals over $[0,2\pi)$ instead of $l(\varepsilon)$ (see (43)), that is, to
obtain a spectral expansion from (22) we divide the set $\mathbb{Z}$ into
three pairwise disjoint subsets $\mathbb{Z}\backslash\mathbb{T},$
$\mathbb{T}\backslash\mathbb{S},$ and $\mathbb{S}$ and consider separately the
integrals of $a_{k}(t)\Psi_{k,t}(x)$ over $l(\varepsilon)$ when the index $k$
varies through these subsets. For this first we prove the following.

\begin{lemma}
Let $f\in E$ and $x\in(-\infty,\infty).$

$(a)$ If $k\in\mathbb{Z}\backslash\mathbb{T}$ then (44) holds.

$(b)$ The following equality holds
\begin{equation}
\sum_{k\in\mathbb{T}}\int\limits_{l(\varepsilon)}a_{k}(t)\Psi_{k,t}%
(x)dt=\int\limits_{[0,2\pi)}\sum_{k\in\mathbb{T}}a_{k}(t)\Psi_{k,t}(x)dt.
\end{equation}

$(c)$ If $k\in\mathbb{T}\backslash\mathbb{S}$, then $a_{k}(t)\Psi_{k,t}(x)$ is
integrable with respect to $t$ over $[0,2\pi).$

$(d)$ The expression $S(x,t),$ defined in (55) is integrable with respect to
$t$ over $[0,2\pi)$.
\end{lemma}

\begin{proof}
$(a)$ By Proposition 5$(c)$, if $k\in\mathbb{Z}\backslash\mathbb{T}$, then
\ $\lambda_{k}(t_{j})$ is a simple eigenvalue for $j=1,2,...,s.$ Therefore
(44) holds due to Proposition 3$(b)$.

$(b)$ Since $\mathbb{T}_{j}$ is the union of the pairwise disjoint sets
$\mathbb{T}(1,j),$ $\mathbb{T}(2,j),...,\mathbb{T}(s_{j},j)$ (see (28)), by
(31) we have
\begin{equation}
\sum_{k\in\mathbb{T}_{j}}\int\limits_{\gamma(t_{j},\varepsilon)}a_{k}%
(t)\Psi_{k,t}(x)dt=\int\limits_{(t_{j}-\varepsilon,t_{j}+\varepsilon)}%
\sum_{k\in\mathbb{T}_{j}}a_{k}(t)\Psi_{k,t}(x)dt.
\end{equation}
On the other hand, by Proposition 5$(c)$ if $k\in\mathbb{T}\backslash
\mathbb{T}_{j}$ then $\lambda_{k}(t)$ is a simple eigenvalue in $U(t_{j}%
,\varepsilon_{j}).$ It implies that $a_{k}(t)\Psi_{k,t}(x)$ analytically
depend on $t$ in $U(t_{j},\varepsilon_{j})$ and hence
\[
\int\nolimits_{\gamma(t_{j},\varepsilon)}a_{k}(t)\Psi_{k,t}(x)dt=\int
\nolimits_{(t_{j}-\varepsilon,t_{j}+\varepsilon)}a_{k}(t)\Psi_{k,t}(x)dt
\]
for $k\in\mathbb{T}\backslash\mathbb{T}_{j}.$ Therefore in (57) one can
replace $\mathbb{T}_{j}$ by $\mathbb{T}$ and get
\begin{equation}
\sum_{k\in\mathbb{T}}\int\limits_{\gamma(t_{j},\varepsilon)}a_{k}(t)\Psi
_{k,t}(x)dt=\int\limits_{(t_{j}-\varepsilon,t_{j}+\varepsilon)}\sum
_{k\in\mathbb{T}}a_{k}(t)\Psi_{k,t}(x)dt
\end{equation}
for all $j=1,2,...,s.$ Now, taking into account that the sets in (42) does not
contain the elements of $A$ we obtain that $a_{k}(t)\Psi_{k,t}(x)$ for all
$k\in\mathbb{Z}$ is integrable over the sets in (42). Then using (43) we get
(56). The proof of $(c)$ follows from Proposition 5$(b)$ and Theorem 9.

$(d)$ Since $S(x,t)$ is integrable over the sets in (42), it is enough to show
that \ $S(x,t)$ is an integrable function on $(t_{j}-\varepsilon
,t_{j}+\varepsilon)$ for arbitrary $j.$ Using (28) and taking into account
that $\mathbb{S}(v,j)\subset\mathbb{T}(v,j)$ for all $v$ we conclude that
$\mathbb{S}(v,j)\cap\mathbb{S}(i,j)=\emptyset$ for all $i\neq v$ . Therefore
\[
S_{j}(x,t)=\sum\nolimits_{v=1}^{m_{j}}S_{v,j}(x,t)
\]
and by Proposition 4, $S_{j}(x,t)$ is a integrable over $(t_{j}-\varepsilon
,t_{j}+\varepsilon).$ Now consider the function
\[
S(x,t)-S_{j}(x,t)=\sum\nolimits_{k\in\mathbb{S}\backslash\mathbb{S}_{j}}%
a_{k}(t)\Psi_{k,t}(x).
\]
By Proposition 5$(b)$, if $k\notin\mathbb{S}_{j}$, then $\frac{1}{\alpha_{k}}$
and hence by Theorem 9 $a_{k}(t)\Psi_{k,t}(x)$ is integrable in $(t_{j}%
-\varepsilon,t_{j}+\varepsilon)$ for all $x.$ Thus $S_{j}(x,t)$ and
$S(x,t)-S_{j}(x,t)$ and hence $S(x,t)$ are integrable in $(t_{j}%
-\varepsilon,t_{j}+\varepsilon)$ for all $j$
\end{proof}

Now we ready to prove the following.

\begin{theorem}
If $L$ has the singular quasimomenta $t_{1},t_{2},...,t_{m}$, then for each
$f\in E$ the following spectral expansion holds
\begin{equation}
f=\frac{1}{2\pi}\int_{[0,2\pi)}\left(
{\textstyle\sum\limits_{k\in\mathbb{S}}}
a_{k}(t)\Psi_{k,t}(x)\right)  dt+\frac{1}{2\pi}\sum_{k\in\mathbb{Z}%
\backslash\mathbb{S}}\int_{[0,2\pi)}a_{k}(t)\Psi_{k,t}dt,
\end{equation}
where $\mathbb{S}$ consist of finite number of indices and
\begin{equation}
\int_{\lbrack0,2\pi)}%
{\textstyle\sum\limits_{k\in\mathbb{S}}}
a_{k}(t)\Psi_{k,t}(x)dt=\lim_{\delta\rightarrow0}\sum_{k\in\mathbb{S}}%
\int\nolimits_{I(\delta)}a_{k}(t)\Psi_{k,t}dt,
\end{equation}
$I(\delta)=[-h,2\pi-h]\backslash%
{\textstyle\bigcup\nolimits_{j=1}^{m}}
(t_{j}-\delta,t_{j}+\delta),$ that is, $I(\delta)$ is obtained from
$[-h,2\pi-h]$ by deleting the $\delta\in(0,\varepsilon)$ neighborhood of the
SQ. The series in (59) converges in the norm of $L_{2}(a,b)$ for every
$a,b\in\mathbb{R}$.
\end{theorem}

\begin{proof}
Using Lemma 1$(a)$ and (22), we obtain
\begin{equation}
f(x)=\frac{1}{2\pi}\sum_{k\in\mathbb{T}}\int\nolimits_{l(\varepsilon)}%
a_{k}(t)\Psi_{k,t}(x)dt+\frac{1}{2\pi}\sum_{k\in\mathbb{Z}\backslash
\mathbb{T}}\int_{[0,2\pi)}a_{k}(t)\Psi_{k,t}(x)dt,
\end{equation}
where $\mathbb{T}$ consist of finite number of indices. For the first
summation in the right side of (61) using (56) and taking into account Lemma
1$(c)$ we get (59).

It is clear that for each $k\in\mathbb{S}$ and for each fixed $\delta
\in(0,\varepsilon)$ the function $a_{k}(t)\Psi_{k,t}$ is integrable with
respect to $t$ on $I(\delta)$. Therefore for any $\delta\in(0,\varepsilon)$ we
have
\begin{equation}
\int_{\lbrack-h,2\pi-h]}S(x,t)dt=\sum_{k\in\mathbb{S}}\int\nolimits_{I(\delta
)}a_{k}(t)\Psi_{k,t}(x)dt+\int\nolimits_{[-h,2\pi-h]\backslash I(\delta
)}S(x,t)dt,
\end{equation}
where $S(x,t)$ is defined in (55). Since the measure of $[-h,2\pi-h]\backslash
I(\delta)$ tends to zero as $\delta\rightarrow0$ and $S(x,t)$ is an integrable
function in $[-h,2\pi-h]$ (see Lemma 1$(d)$ and (45)) the last integral in
(62) tends to zero as $\delta\rightarrow0$. Therefore (60) follows from (62)
\end{proof}

\begin{remark}
In the spectral expansion theorems (see Theorem 3) we prefer to use the
integrals of $a_{k}(t)\Psi_{k,t}(x)$ with respect to $t$ over $[0,2\pi),$
since the set $\left\{  \lambda_{k}(t):t\in\lbrack0,2\pi)\right\}  $ is the
traditional $k$th band $\Gamma_{k}$ (see Remark 2). Hence in (59) we have sum
of the integrals over the bands which is natural for the spectral expansion of
the periodic differential operators. Note that if for the spectral expansion
we wish to write the integrals over $[0,2\pi),$ then the huddling over
$\mathbb{S}$ (see (59) and (60)), in general, is necessary and one can not
divide the set $\mathbb{S}$ into two disjoint subset $\mathbb{S}^{\prime}$ and
$\mathbb{S}\backslash\mathbb{S}^{\prime}$ such that
\begin{equation}%
{\displaystyle\sum\nolimits_{k\in\mathbb{S}^{\prime}}}
a_{k}(t)\Psi_{k,t}(x)\text{ }\And\text{ }%
{\displaystyle\sum\nolimits_{k\in\mathbb{S}\backslash\mathbb{S}^{\prime}}}
a_{k}(t)\Psi_{k,t}(x)
\end{equation}
are integrable over $[0,2\pi)$. Indeed, if $\mathbb{S}_{1}=\left\{
1,2\right\}  ,$ $\mathbb{S}_{2}=\left\{  2,3\right\}  ,...,\mathbb{S}%
_{s}=\left\{  s,s+1\right\}  ,$ then by (54), Proposition 5$(b)$ and Theorem
9, $\mathbb{S}=\left\{  1,2,...,s+1\right\}  $ and $a_{k}(t)\Psi_{k,t}(x)$ is
nonintegrable with respect to $t$ over $[0,2\pi)$ for $k=1,2,...,s+1$ and for
some $f\in E$, $x\in\lbrack0,1].$ It is clear that for any proper subset
$\mathbb{S}^{\prime}$ of $\mathbb{S}$ the expressions in (63) are not
integrable over $[0,2\pi).$

However, if we consider the integral over $[0,2\pi)$ as sum of integrals over
$I(\delta)$ and

$(t_{j}-\delta,t_{j}+\delta)$ for $j=1,2,...,m,$ where $I(\delta)$ is defined
in Theorem 3, then the first integral in (59) can be written in the form
\begin{equation}
\int\limits_{\lbrack0,2\pi)}%
{\displaystyle\sum\limits_{k\in\mathbb{S}}}
a_{k}(t)\Psi_{k,t}dt=\sum_{k\in\mathbb{S}}\int\limits_{I(\delta)}a_{k}%
(t)\Psi_{k,t}dt+%
{\displaystyle\sum\nolimits_{j=1}^{m}}
\int\limits_{(t_{j}-\delta,t_{j}+\delta)}%
{\displaystyle\sum\limits_{k\in\mathbb{S}}}
a_{k}(t)\Psi_{k,t}dt.
\end{equation}
On the other hand, taking into account that the functions
\[
\sum\nolimits_{k\in\mathbb{S}(v,j)}a_{k}(t)\Psi_{k,t}(x)\text{ },\text{ }%
{\displaystyle\sum\nolimits_{k\in\mathbb{S}_{j}}}
a_{k}(t)\Psi_{k,t}\text{ }\And%
{\displaystyle\sum\nolimits_{k\in\mathbb{S}}}
a_{k}(t)\Psi_{k,t}%
\]
are integrable in $(t_{j}-\delta,t_{j}+\delta)$ (see Proposition 4 and the
proof of Lemma 1$(d)$) and using the relations $\mathbb{S}_{j}=\cup
_{i=1}^{m_{j}}\mathbb{S}(i,j)$ $\And$ $\mathbb{S}(i,j)\cap\mathbb{S}(v,j),$
$\forall v\neq i,$ we obtain
\begin{equation}
\int\limits_{(t_{j}-\delta,t_{j}+\delta)}%
{\displaystyle\sum\limits_{k\in\mathbb{S}}}
a_{k}(t)\Psi_{k,t}dt=%
{\displaystyle\sum\limits_{k\in\mathbb{S}\backslash\mathbb{S}_{j}}}
\int\limits_{(t_{j}-\delta,t_{j}+\delta)}a_{k}(t)\Psi_{k,t}dt+\sum
\limits_{i=1}^{m_{j}}\int\limits_{(t_{j}-\delta,t_{j}+\delta)}\left(
{\displaystyle\sum\limits_{k\in\mathbb{S}(i,j)}}
a_{k}(t)\Psi_{k,t}\right)  dt.
\end{equation}
Moreover, by definition of $\mathbb{S}(i,j)$ and Theorem 9 for each
$k\in\mathbb{S}(i,j)$ there exist $f\in E$ and $x\in(-\infty,\infty)$ such
that $a_{k}(t)\Psi_{k,t}(x)$ is not integrable with respect to $t$ over
$(t_{j}-\delta,t_{j}+\delta).$ Therefore in (65) we huddle together the
nonintegrable (over $(t_{j}-\delta,t_{j}+\delta))$ elements. Hence \ for the
considerations of the integrals over $(t_{i}-\delta,t_{i}+\delta)$ the
huddling (the parenthesis in (65)) over $\mathbb{S}(i,j)$ is necessary. It
means that, in the spectral expansions the parenthesis comprising the
functions with indices from $\mathbb{S}(i,j)$ corresponding to ESS
$\Lambda_{j}(t_{i})$ is necessary. Using (64) and (65) one can minimize the
number of terms in the parenthesis by taking the sums over $\mathbb{S}(i,j)$
corresponding to only one ESS $\Lambda_{j}(t_{i})$ and the integrals over
$I(\delta)$ and $(t_{i}-\delta,t_{i}+\delta)$ for $i=1,2,...,s.$ In Theorem 3
to avoid the complicated notations and to get a spectral expansion suitable
for the periodic differential operators we prefer to use the integrals over
$[0,2\pi).$ That is why, in (59) we use the parenthesis comprising the
functions with indices from $\mathbb{S}$ corresponding to all ESS.
\end{remark}

\section{Spectral Expansion for Even Order}

The case $n=2\mu$ is more complicated than the case $n=2\mu+1$ due to the
following. In the case of odd order the numbers of ESS and SQ are finite and
the operator $L$ has no ESS at infinity that easify the investigations (see
Proposition 3$(c)$). In the big contrary of the odd order case, in the case of
even order we, in general, meet with the both complexities:

$(a)$ the existence of the infinite number of the ESS and SQ,

$(b)$ the existence of the ESS at infinity.

The complexity $(a)$ occurs because, in general, the set $A\cap\lbrack0,2\pi)$
defined in Remark 1 contains infinite number of points. Fortunately, it
follows from uniform asymptotic formulas theorem that the possible
accumulation points of the set $A\cap Q$ are $0$ and $\pi$. Hence the set
$A\cap Q_{h}$ is finite and the sets $A\cap\left\{  t:|t|\leq h\right\}  $ and
$A\cap\left\{  t:|t-\pi|\leq h\right\}  $ are not finite, in general.
Therefore we need to investigate in detail the eigenvalues and eigenfunctions
of $L_{t}$ for the cases $|t|\leq h$ and $|t-\pi|\leq h.$ Moreover, the
complexity $(b)$ is connected with the fact that there may exists a sequence
of pairs $\left\{  n_{k},t_{k}\right\}  $ such that $\left\vert n_{k}%
\right\vert \rightarrow\infty$ and $\alpha_{n_{k}}(t_{k})\rightarrow0$ as
$k\rightarrow\infty$ and the sequence $\left\{  t_{k}\right\}  $ may have only
two accumulation points $0$ and $\pi.$ It again shows the importance of the
detail considerations about $0$ and $\pi.$

Thus, first of all, let us consider the cases $|t|\leq h$ and $|t-\pi|\leq h.$
For this we follow the arguments given in [6] for the proof of the asymptotic
formulas for the eigenvalues and eigenfunction, and take into consideration
the uniformity with respect to $t.$ First, let us introduce some notations.
Let $T$ be a domain of the complex plane such that if $\rho\in T,$ then the
inequalities
\begin{equation}
\operatorname{Re}(\rho+c_{1})\omega_{1}\leq\operatorname{Re}(\rho+c_{1}%
)\omega_{2}\leq...\leq\operatorname{Re}(\rho+c_{1})\omega_{n}%
\end{equation}
hold, for a suitable ordering of the $n$th roots $\omega_{1},\omega_{2}%
,\ldots,\omega_{n}$ of $-1$ (see p. 45 of [6]), where $c_{1}$ is a constant.
Let $h$ and $r$ be constant satisfying
\begin{equation}
0<h<1/32\text{ }\And\text{\ }1/4\leq r\leq1/2.
\end{equation}
Assume that the constant $c_{1}$ in (66) is chosen so that $\frac{1}%
{\omega_{\mu}}\rho\in T$ if $\rho$ belongs to the disks%
\begin{equation}
\left\{  z\in\mathbb{C}:\left\vert z-(i2k\pi)\right\vert <1\right\}
\end{equation}
for the large positive values of $k.$ From the arguments of [6] by the simple
estimations we obtain the following statements.

\begin{theorem}
If $n=2\mu$ then there exists a positive integer $N_{h}(0)$ independent of $t$
such that for $k>N_{h}(0)$ the disk%
\begin{equation}
\left\{  z\in\mathbb{C}:\left\vert z-(i2k\pi)^{n}\right\vert <\tfrac{1}%
{4}n(2\pi k+\tfrac{1}{4})^{n-1}\right\}
\end{equation}
contain only two eigenvalues (counting multiplicities) denoted by $\lambda
_{k}(t)$ and $\lambda_{-k}(t)$ of the operators $L_{t}$ for $|t|\leq h.$
Moreover, the washer%
\begin{equation}
\left\{  \tfrac{1}{4}n(2k\pi+\tfrac{1}{4})^{n-1}\leq\left\vert z-(i2k\pi
)^{n}\right\vert \leq\tfrac{1}{2}n(2k\pi+\tfrac{1}{4})^{n-1}\right\}
\end{equation}
for $k>N_{h}(0)$ does not contain the eigenvalues of $L_{t}$ for $|t|\leq h.$
\end{theorem}

\begin{proof}
It was proved in [6] (see (69) in p.71 of [6]) that to investigate the large
roots of the characteristic equations $\Delta(\rho,t)=0$ it is enough to
consider the equation
\begin{equation}
\theta_{1}(e^{\rho\omega_{\mu}}-\xi^{^{\prime}})(e^{\rho\omega_{\mu}}%
-\xi^{^{\prime\prime}})=O\left(  \rho^{-1}\right)
\end{equation}
in a fixed domain $T,$ where in the case of boundary conditions (3) by direct
calculation one can verify that $\theta_{1}=(-e^{it})^{\mu}B,$ $\xi^{^{\prime
}}=e^{it},$ $\xi^{^{\prime\prime}}=e^{-it}$ (see for example (7) of [12]) and
$B$ is Vandermonde determinant for $\omega_{1},\omega_{2},\ldots,\omega_{n}.$.
Therefore (71) has the form
\begin{equation}
(e^{z}-e^{it})(e^{z}-e^{-it})=O\left(  \rho^{-1}\right)  ,
\end{equation}
where $z=\rho\omega_{\mu}$ and there exists a constant $c_{2}$ such that
$\left\vert O\left(  \rho^{-1}\right)  \right\vert <\left\vert c_{2}\rho
^{-1}\right\vert $ for $|t|\leq h.$ Using the Taylor series of $e^{z}$ at
$2\pi ki$ we get $\left\vert e^{z}-1\right\vert >r/2$ if $z$ belongs to the circle

$\gamma_{r}=:\left\{  z\in\mathbb{C}:\left\vert z-(i2k\pi)\right\vert
=r\right\}  ,$ since $r\leq1/2$ (see (67)). If $|t|\leq h,$ then using (69)
and the Maclaurin's series of $e^{x}$ we obtain $\left\vert e^{\pm
it}-1\right\vert <\left\vert 2t\right\vert \leq2h<r/4.$ These inequalities
imply that
\begin{equation}
\left\vert e^{z}-e^{\pm it}\right\vert >r/4,\text{ }\left\vert (e^{z}%
-e^{it})(e^{z}-e^{-it})\right\vert >r^{2}/16
\end{equation}
if $z\in\gamma_{r}.$ It yields the claim (\textit{Claim 1)} that (72) has no
roots on the washer%
\[
1/4\leq\left\vert z-i2k\pi\right\vert \leq1/2
\]
for large values of $k.$ Moreover, taking into account that for $|t|\leq
h<1/32$ the equation

$(e^{z}-e^{it})(e^{z}-e^{-it})=0$ has two roots in the disk
\begin{equation}
\{\xi\in\mathbb{C}:\left\vert \xi-i2\pi k\right\vert <1/4\}
\end{equation}
and using Rouche's theorem we get the claim (\textit{Claim 2)} that the disk
(74) contains only $2$ roots of (72). Thus repeating the arguments of the
proof of Theorem 2 in pages 70-74 of [6] from the \textit{Claim 1} and
\textit{Claim 2 }we obtain the proof of the theorem
\end{proof}

In the same way we prove the following.

\begin{theorem}
If $n=2\mu$ then there exists an integer $N_{h}(\pi)$ independent of $t$ such
that for $k>N_{h}(\pi)$ the disk%
\begin{equation}
\left\{  \left\vert z-(i(2k\pi+\pi))^{n}\right\vert <\tfrac{1}{4}n(2\pi
k+\pi+\tfrac{1}{4})^{n-1}\right\}
\end{equation}
contain only two eigenvalues (counting multiplicities) denoted by $\lambda
_{k}(t)$ and $\lambda_{-(k+1)}(t)$ of the operators $L_{t}$ for $|t-\pi|\leq
h.$ Moreover, the washer%
\[
\left\{  \tfrac{1}{4}n(2\pi k+\pi+\tfrac{1}{4})^{n-1}\leq\left\vert
z-(i2k\pi+i\pi)^{n}\right\vert \leq\tfrac{1}{2}n(2\pi k+\pi+\tfrac{1}%
{4})^{n-1}\right\}
\]
does not contain the eigenvalues of $L_{t}$ for $|t-\pi|\leq h.$
\end{theorem}

\begin{remark}
Consider the family of operators
\[
L_{t,z}=L_{t}(0)+z(L_{t}-L_{t}(0)),\text{ }0\leq z\leq1,\text{ }|t|\leq h,
\]
where $L_{t}(0)$ denotes the case when all coefficients of (1) are zero.
Repeating the proof of Theorem 4 one can readily see that there exist
$N_{h}(0),$ independent of $z\in\lbrack0,1]$ and $|t|\leq h,$ so that Theorems
4 continues to hold for the operators $L_{t,z}$. Therefore, there exists a
closed curve $\Gamma(0)$ such that:

$(a)$ the curve $\Gamma(0)$ lies in the resolvent set of $L_{t,z}$ for
$z\in\lbrack0,1]$ and $|t|\leq h.$

$(b)$ all eigenvalues of $L_{t,z}$ for $z\in\lbrack0,1]$ and $|t|\leq h$ that
do not lie in (69) for $k>N_{h}(0)$ belong to the set enclosed by $\Gamma(0).$

Therefore, taking into account that the family $L_{t,z}$ is halomorphic with
respect to $z,$ we obtain that the number of eigenvalues of the operators
$L_{t,0}=L_{t}(0)$ and $L_{t,1}=L_{t}$ lying inside of $\Gamma(0)$ are the
same. It means that apart from the eigenvalues $\lambda_{k}(t),$ where
$|k|>N_{h}(0),$ there exist $(2N_{h}(0)+1)$ eigenvalues of the operator
$L_{t}$ for $|t|\leq h.$denoted by
\begin{equation}
\lambda_{k}(t)\text{ }:\text{ \ }k\in\mathbb{N}(0)
\end{equation}
and lying in $\Gamma(0),$ where $\mathbb{N}(0)=\left\{  0,\pm1,...,\pm
N_{h}(0)\right\}  .$

By the same arguments we obtain that apart from the eigenvalues $\lambda
_{k}(t)$ and $\lambda_{-(k+1)}(t)$ \ lying in (75) for $k>N_{h}(\pi)$ there
exist $(2N_{h}(\pi)+2)$ eigenvalues the operator $L_{t}$ for $|t-\pi|\leq h$
denoted by
\begin{equation}
\lambda_{k}(t)\text{: }k\in\mathbb{N}(\pi),
\end{equation}
where $\mathbb{N}(\pi)=\left\{  0,\pm1,...,\pm N_{h}(\pi),-(N_{h}%
(\pi)+1)\right\}  ,$ and there exist a closed curve $\Gamma(\pi)$ which
contains inside only the eigenvalues (77). Thus in any case we numerate the
eigenvalues of $L_{t}$ by the elements of $\mathbb{Z}.$ Moreover using uniform
asymptotic formulas theorem (see introduction), Theorems 4 and 5 and repeating
the proof of (15) and (16) of [15] one can show that the eigenvalues of
$L_{t}$ can be numbered (counting the multiplicity) by elements of
$\mathbb{Z}$ such that for each $k$ the function $\lambda_{k}(t)$ is
continuous on $[0,2\pi)$ and for $|k|>N(h)$ formula (6) holds (for the
continuous numeration one can look also Lemma 3.1 of [13]).
\end{remark}

Theorem 4 shows that for the large values of $k$ the eigenvalue $\lambda
_{k}(t)$ of $L_{t}$ for $|t|\leq h$ is close to the eigenvalues $\left(
\pm2k\pi i+it\right)  ^{n}$ of $L_{t}(0)$ and far from the other eigenvalues
$L_{t}(0).$ More precisely, Theorem 4 implies that
\[
|\lambda_{\pm k}(t)-\left(  2\pi pi+it\right)  ^{n}%
|>(||k|-|p||)(|k|+|p|)^{n-1}%
\]
for $p\neq\pm k,$ $|t|\leq h$ and $k>N_{h}(0)$. Using this one can easily
verify that the formulas%
\begin{equation}
\sum_{p:p\neq\pm k}\left(  \frac{\mid p\mid^{n-2}}{\left\vert \lambda_{\pm
k}(t)-(2\pi pi+it)^{n}\right\vert }\right)  =O(k^{-1}\ln\left\vert
k\right\vert ),
\end{equation}%
\begin{equation}
\sum_{p:p\neq\pm k}\left(  \frac{\mid p\mid^{n-2}}{\left\vert \lambda_{\pm
k}(t)-(2\pi pi+it)^{n}\right\vert }\right)  ^{2}=O(k^{-2})
\end{equation}
hold uniformly with respect to $|t|\leq h.$ Using (79) we obtain the following.

\begin{theorem}
The normalized eigenfunction $\Psi_{k,t}$ of $L_{t}$ satisfies the following,
uniform with respect to $t$ in $\left\{  t\in\left(  \mathbb{C}\backslash
A\right)  :|t|\leq h\right\}  ,$ asymptotic formula
\[
\Psi_{k,t}(x)=e^{itx}(u_{k,k}(t)e^{i2\pi kx}+u_{k,-k}(t)e^{-i2\pi kx}%
+h_{k,t}(x)),\text{ }\left\Vert h_{k,t}\right\Vert =O(k^{-1}),
\]
where $u_{k,j}(t)=(\Psi_{k,t},e^{i(2\pi j+\overline{t})x}),$ $(h_{k,t},e^{\pm
i2\pi kx})=0,$ $\sup\limits_{x\in\lbrack0,1]}|h_{k,t}(x)|=O\left(  k^{-1}%
\ln\left\vert k\right\vert \right)  .$
\end{theorem}

\begin{proof}
For the proof we use the formula
\begin{equation}
(\lambda_{k}(t)-(2p\pi i+it)^{n})\left(  \Psi_{k,t},e^{(2p\pi i+i\overline
{t})}\right)  =\left(  p_{2}\Psi_{k,t}^{(n-2)}+\cdots+p_{n}\Psi_{k,t}%
,e^{(2p\pi i+i\overline{t})x}\right)
\end{equation}
which can be obtained from $\Psi_{k,t}^{(n)}+p_{2}\Psi_{k,t}^{(n-2)}%
+\cdots+p_{n}(x)\Psi_{k,t}=\lambda_{k}(t)\Psi_{k,t}$ by multiplying by
$e^{(2p\pi i+i\overline{t})x}$ . Using the integration by part and (2) one can
easily verify that there exists a constant $c_{3}$ such that
\begin{equation}
\left\vert \left(  p_{2}\Psi_{k,t}^{(n-2)}+\cdots+p_{n}\Psi_{k,t},e^{(2p\pi
i+i\overline{t})x}\right)  \right\vert <c_{3}p^{n-2}.
\end{equation}
Therefore using (79) and (78) we obtain
\[%
{\textstyle\sum\limits_{p:p\neq\pm k}}
\mid\left(  \Psi_{k,t},e^{(2p\pi i+i\overline{t})x}\right)  \mid^{2}%
=O(k^{-2}),\text{ }%
{\textstyle\sum\limits_{p:p\neq\pm k}}
\mid\left(  \Psi_{k,t},e^{(2p\pi i+i\overline{t})x}\right)  \mid=O(k^{-1}%
\ln\left\vert k\right\vert )
\]
Now decomposing $\Psi_{k,t}$ by the basis $\{e^{(2k\pi i+it)x}:k\in
\mathbb{Z}\}$, we get the proof of the theorem
\end{proof}

Instead of Theorem 4 using Theorem 5 and repeating the proof of Theorem 6 we obtain

\begin{theorem}
The normalized eigenfunction $\Psi_{k,t}$ of $L_{t}$ satisfies the following,
uniform with respect to $t$ in $\left\{  t\in\left(  \mathbb{C}\backslash
A\right)  :|t-\pi|\leq h\right\}  $ asymptotic formula
\begin{equation}
\Psi_{k,t}(x)=e^{itx}\left(  u_{k,k}(t)e^{i2\pi kx}+u_{k,-(k+1)}%
(t)e^{-i2\pi(k+1)x}+\phi_{k,t}(x)\right)  ,
\end{equation}
where $\left\Vert \phi_{k,t}\right\Vert =O(k^{-1}),$ $\sup\limits_{x\in
\lbrack0,1]}|\phi_{k,t}(x)|=O\left(  k^{-1}\ln\left\vert k\right\vert \right)
,$ $(\phi_{k,t},e^{i2\pi jx})=0$ for $j=k,-(k+1).$
\end{theorem}

Repeating the proof of Theorem 6 we obtain

\begin{theorem}
The normalized eigenfunction $\Psi_{k,t}$ of $L_{t}$ satisfies the following,
uniform with respect to $t$ in $[0,2\pi]\backslash A,$ asymptotic formula
\begin{equation}
\Psi_{k,t}(x)=e^{itx}(u_{k,k}(t)e^{i2\pi kx}+u_{k,-k}(t)e^{-i2\pi
kx}+u_{k,-(k+1)}(t)e^{-i2\pi(k+1)x}+\varphi_{k,t}(x)),
\end{equation}
where $\sup\limits_{x\in\lbrack0,1]}|\varphi_{k,t}(x)|=O\left(  k^{-1}%
\ln\left\vert k\right\vert \right)  .$
\end{theorem}

Moreover using (80) and (81) by direct calculation we obtain the following
lemma which plays the crucial role in the proof of the next theorem.

\begin{lemma}
$(a)$ For each $k\in\mathbb{Z}$ there exists $l\in\mathbb{N}$ such that
\begin{align}
\Psi_{k,t}(x) &  =P_{l,t}(x)+h_{l,t}(x),\text{ }\left\vert h_{l,t}%
(x)\right\vert <1/4,\\
P_{l,t}(x) &  =%
{\textstyle\sum\limits_{\left\vert p\right\vert <l}}
\left(  \Psi_{k,t},e^{i(2\pi p+t)x}\right)  e^{i(2\pi p+t)x}%
\end{align}
for all $x\in\lbrack0,1]$ and $t\in\lbrack0,2\pi)\backslash A(k),$ where
$A(k)$ is defined in Remark 2.

$(b)$ For each $k\in\mathbb{Z}$ there exists a constant $c_{4}$ such that
\[
\left\vert \Psi_{k,t}(x)\right\vert <c_{4},\text{ }\left\vert \Psi_{k,t}%
^{\ast}(x)\right\vert <c_{4},\text{ }\forall x\in\lbrack0,1],\text{ }%
t\in\left(  \lbrack0,2\pi)\backslash A(k)\right)  .
\]

$(c)$ Suppose that $t_{0}\in A(k)$ and $(t_{0},t_{0}+\delta)\cap
A(k)=\varnothing.$ Then for each $t\in(t_{0},t_{0}+\delta)$ there exists
positive integer $j$ such that%
\begin{equation}
\left\vert \Psi_{k,t}(x)\right\vert >\tfrac{1}{4},\text{ }\forall x\in\left[
\tfrac{j-1}{q},\tfrac{j}{q}\right]  ,
\end{equation}
where $1\leq j\leq q,$ $q\in\mathbb{N}$, $q>8\pi l$ and $l$ is defined in
$(a).$
\end{lemma}

\begin{proof}
$\left(  a\right)  $ It is clear that for each $k\in\mathbb{Z}$ there exists
$l\in\mathbb{N}$ such that
\[
|\lambda_{\pm k}(t)-\left(  2\pi pi+it\right)  ^{n}|>|p|^{n},\text{ }%
\forall|p|\geq l,\text{ }t\in\lbrack0,2\pi)
\]
Therefore using (80) and (81) we obtain that if $\left\vert p\right\vert \geq
l,$ then there exists $c_{5}$ such that
\[
\left\vert \left(  \Psi_{k,t},e^{i(2\pi p+t)x}\right)  \right\vert
<c_{5}/p^{2}.
\]
It with the Fourier decomposition of $\Psi_{k,t}$ with respect to the
orthonormal basis $\left\{  e^{i(2\pi p+t)x}:p\in\mathbb{Z}\right\}  $ yields
(84) and (85).

The proof of $(b)$ follows from $(a)$.

$(c)$ Since $\left\Vert \Psi_{k,t}\right\Vert =1,$ for each $t\in(t_{0}%
,t_{0}+\delta)$ there exists $x(t)\in\lbrack0,1]$ such that
\begin{equation}
\left\vert \Psi_{k,t}(x(t))\right\vert \geq1,\text{ }%
{\textstyle\sum\limits_{\left\vert p\right\vert <l}}
\left\vert \left(  \Psi_{k,t},e^{i(2\pi p+t)x}\right)  \right\vert ^{2}\leq1.
\end{equation}
On the other hand, it follows from (84), (85) and (87) that%
\[
\left\vert P_{l,t}(x(t))\right\vert >3/4,\text{ }\left\vert P_{l,t}^{^{\prime
}}(x)\right\vert <2\pi l,
\]
where $P_{k,t}^{^{\prime}}(x)$ is the derivative of $P_{k,t}(x)$ with respect
to $x$. These inequalities imply that
\[
\left\vert P_{l,t}(x)\right\vert >1/2,\text{ }\forall x\in\lbrack
x(t)-1/q,x(t)+1/q],
\]
since $q>8\pi l.$ Therefore using (84) we obtain
\[
\left\vert \Psi_{k,t}(x)\right\vert >1/4,\text{ }\forall x\in\lbrack
x(t)-1/q,x(t)+1/q].
\]
Thus if $x(t)\in\lbrack\tfrac{j-1}{q},\tfrac{j}{q}],$ then (86) holds
\end{proof}

\begin{theorem}
Let $\lambda_{k}(t)$\ be multiple and simple eigenvalue for $t=t_{0}$ and
$t\in D(t_{0},\delta),$ where $D(t_{0},\delta):=(t_{0}-\delta,t_{0}%
+\delta)\backslash\left\{  t_{0}\right\}  ,$ respectively. Then the integral
\begin{equation}
\int\nolimits_{(t_{0}-\delta,t_{0}+\delta)}a_{k}(t)\Psi_{k,t}(x)dt
\end{equation}
exists for all $x\in(-\infty,\infty)$ and $f\in E$ if and only if $\frac
{1}{\alpha_{k}}\in L_{1}(t_{0}-\delta,t_{0}+\delta).$
\end{theorem}

\begin{proof}
Using Condition 1 and Lemma 2$(b)$ we obtain that there exists $c_{6}$ such
that $\left\vert \left(  f_{t},\Psi_{k,t}^{\ast}\right)  \Psi_{k,t}%
(x)\right\vert <c_{6}$ for all $t\in D(t_{0},\delta)$ and $x\in(-\infty
,\infty).$ Therefore if $\frac{1}{\alpha_{k}}$ is integrable over
$(t_{0}-\delta,t_{0}+\delta),$ then the expression $a_{k}(t)\Psi_{k,t}(x)$
(see (23)) is also integrable over $(t_{0}-\delta,t_{0}+\delta)$ for all
$x\in(-\infty,\infty)$ and $f\in E.$

Now let $\frac{1}{\alpha_{k}}\notin L_{1}(t_{0}-\delta,t_{0}+\delta).$ Without
loss of generality suppose that $\frac{1}{\alpha_{k}}\notin L_{1}(t_{0}%
,t_{0}+\delta)$ which implies that
\begin{equation}
\int\nolimits_{(t_{s},t_{s}+\delta)}\left\vert \alpha_{k}(t)\right\vert
^{-1}dt\rightarrow\infty
\end{equation}
as $t_{s}\rightarrow t_{0}$ and $t_{s}>t_{0}.$ Introduce the notation
\begin{align*}
A_{j}  &  =\left\{  t\in(t_{0},t_{0}+\delta):\text{ }\left\vert \Psi
_{k,t}(x)\right\vert >\tfrac{1}{4},\text{ }\forall x\in\left[  \tfrac{j-1}%
{q},\tfrac{j}{q}\right]  \right\}  ,\\
B_{p}  &  =\left\{  t\in(t_{0},t_{0}+\delta):\text{ }\left\vert \left(
\Psi_{k,t}^{\ast},e^{i(2\pi p+t)x}\right)  \right\vert >\tfrac{1}{2l}\right\}
.
\end{align*}
By (2) formulas (84) and (85) holds for $\Psi_{k,t}^{\ast}$ too. Therefore
using Parseval equality we obtain that for each $t\in(t_{0},t_{0}+\delta)$
there exists $p$ such that $\left\vert \left(  \Psi_{k,t}^{\ast},e^{i(2\pi
p+t)x}\right)  \right\vert >1/2l.$ Therefore by Lemma 2$(c)$ the union of
$A_{j}\cap B_{p}\ $\ for $j=1,2,...,q$ and $\left\vert p\right\vert <l$ is
$(t_{0},t_{0}+\delta).$ Thus by (89) there exists $j$ and $p$ such that
\begin{equation}
\int\nolimits_{(t_{s},t_{s}+\delta)\cap\left(  A_{j}\cap B_{p}\right)
}\left\vert \alpha_{k}(t)\right\vert ^{-1}dt\rightarrow\infty
\end{equation}
as $t_{s}\rightarrow t_{0}.$ By the definition of $E$ we have $\Upsilon
^{-1}\left(  e^{i(2\pi p+t)x}\right)  \in E,$ where $\Upsilon^{-1}$ is the
inverse \ Gelfand transform defined by (14). Thus for each fixed $x_{0}%
\in\left[  \tfrac{j-1}{q},\tfrac{j}{q}\right]  $ taking
\begin{equation}
f=\Upsilon^{-1}\left(  e^{i(2\pi p+t)x}\right)
\end{equation}
using (23) and definitions of $A_{j}$ and $B_{p}$ we obtain that
\[
\left\vert a_{k}(t)\Psi_{k,t}(x_{0})\right\vert =\frac{1}{\left\vert
\alpha_{k}(t)\right\vert }\left\vert \Psi_{k,t}(x_{0})\right\vert ^{2}%
>\frac{1}{16l\left\vert \alpha_{k}(t)\right\vert },\text{ }\forall t\in\left(
A_{j}\cap B_{p}\right)  .
\]
It with (90) implies that (88) does not exist for $f$ defined in (91) and for
$x=x_{0}$
\end{proof}

Now we are ready to construct the curve of integration $l$ (see (22)) for the
even order case. Since in the case $n=2\mu$ the numbers $0$ and $\pi$ are
accumulation points of $A\cap Q,$ we consider the set $A^{\prime}=:(A\cap
Q)\cup\{0,\pi\},$ defied in Remark 1, and choose the curve of integration so
that it pass over the points of $A^{\prime}$ . Namely, we construct $l$ as
follows. Let $h$ be positive number satisfying (67) and such that $\pm h\notin
A,$ $\left(  \pi\pm h\right)  \notin A.$ By Remark 1 the set $A\cap Q_{h\text{
}}$ is finite. Denote the points of $A\cap B(h)$ by $t_{1},t_{2},...,t_{s}$,
where
\begin{align}
B(h)  &  =[0,2\pi)\cap Q_{h\text{ }}=[h,\pi-h]\cup\lbrack\pi+h,2\pi-h],\\
h  &  <t_{1}<t_{2}<\cdot\cdot\cdot<t_{p}<\pi-h<\pi+h<t_{p+1}<t_{p+2}%
<\cdot\cdot\cdot<t_{s}<2\pi-h.
\end{align}
Let $\varepsilon$ be positive number such that $2\varepsilon$ is less than the
following numbers
\begin{equation}
h,t_{1}-h,\pi-h-t_{p},\text{ }t_{p+1}-\pi-h,\text{ }2\pi-h-t_{s},t_{j+1}-t_{j}%
\end{equation}
for $j\in\left\{  1,2,...,s-1\right\}  \backslash p,$ that is, $\varepsilon$
is less than half of the distance between the neighboring points of (93). Let
$C(h,\varepsilon)\subset Q_{h}$ be a curve joining the points $-h$ and
$2\pi-h$ and consisting of the intervals
\begin{equation}
\lbrack h,t_{1}-\varepsilon),(t_{j}+\varepsilon,t_{j+1}-\varepsilon
),(t_{p}+\varepsilon,\pi-h),(\pi+h,t_{p+1}-\varepsilon),(t_{s}+\varepsilon
,2\pi-h]
\end{equation}
for $j\in\left\{  1,2,...,s-1\right\}  \backslash\left\{  p\right\}  $ and
semicircles
\begin{align}
\gamma(0,h)  &  =\{\left\vert t\right\vert =h,\operatorname{Im}t\geq0\},\text{
}\gamma(\pi,h)=\{\left\vert t-\pi\right\vert =h,\operatorname{Im}t\geq0\},\\
\gamma(t_{j},\varepsilon)  &  =\{\left\vert t-t_{j}\right\vert =\varepsilon
,\operatorname{Im}t\geq0\}
\end{align}
for $j=1,2,...,s.$ Since $A\cap Q_{\delta}$ is a finite set for any $\delta
>0$, the numbers $h$ and $\varepsilon$ can be chosen so that the curve
$C(h,\varepsilon)$ does not contain the point of the set $A$. Divide
$C(h,\varepsilon)$ into three parts: $\gamma(0,h),$ $\gamma(\pi,h)$ and
$l(h,\varepsilon)=C(h,\varepsilon)\backslash(\gamma(0,h)\cup\gamma(\pi,h)).$
Thus $l(h,\varepsilon)$ consists of the intervals (95) and semicircles (97)
and lies in the domain $Q_{h}\backslash A.$ Since $A\cap Q_{h}$ is a finite
set, $\varepsilon$ can be chosen so that $l(h,\varepsilon)\cap A=\emptyset,$
and the open domain $D(\varepsilon,h)$ enclosed above by the curve
$l(h,\varepsilon)$ and below by $B(h)$ does not contain the point of the set
$A$. The semicircles $\gamma(0,h)$ and $\gamma(\pi,h)$ also lie in
$Q_{h}\backslash A.$

In (20) instead of $l$ taking $C(h,\varepsilon)=l(h,\varepsilon)\cup
\gamma(0,h)\cup\gamma(\pi,h)$ and using (21) we obtain%
\begin{align}
2\pi f  &  =\int\limits_{l(h,\varepsilon)}f_{t}(x)dt+\int\limits_{\gamma
(0,h)}f_{t}(x)dt+\int\limits_{\gamma(\pi,h)}f_{t}(x)dt=\\
&  \int\limits_{l(h,\varepsilon)}\sum\limits_{k\in\mathbb{Z}}a_{k}%
(t)\Psi_{k,t}dt+\int\limits_{\gamma(0,h)}\sum\limits_{k\in\mathbb{Z}}%
a_{k}(t)\Psi_{k,t}dt+\int\limits_{\gamma(\pi,h)}\sum\limits_{k\in\mathbb{Z}%
}a_{k}(t)\Psi_{k,t}dt\nonumber\\
&  =\sum\limits_{k\in\mathbb{Z}}\int\limits_{l(h,\varepsilon)}a_{k}%
(t)\Psi_{k,t}dt+\sum\limits_{k\in\mathbb{Z}}\int\limits_{\gamma(0,h)}%
a_{k}(t)\Psi_{k,t}dt+\sum\limits_{k\in\mathbb{Z}}\int\limits_{\gamma(\pi
,h)}a_{k}(t)\Psi_{k,t}dt.\nonumber
\end{align}
The investigation of the integral over $l(h,\varepsilon)$ in (98) is the
repetition of the investigation of (22) for $l=l(\varepsilon)$ in the case
$n=2\mu+1$, where $l(\varepsilon)$ is defined in (43), due to the followings.

1. The sets $A\cap B(h)$ and $A\cap\lbrack0,2\pi)$ in the cases $n=2\mu$ and
$n=2\mu+1$ respectively are finite. For simplicity of notations in the both
cases we denote these sets by $t_{1},t_{2},...,t_{s}$. The set of multiple
eigenvalues of $L_{t_{j}}$ for $j=1,2,...,s$ is finite and is denoted by
$\Lambda_{1}(t_{j}),$ $\Lambda_{2}(t_{j}),...,\Lambda_{s_{j}}(t_{j})$ as in
the case $n=2\mu+1.$ Moreover the definitions of $\mathbb{S}(v,j)$ and
$\mathbb{T}(v,j)$ for $\Lambda_{v}(t_{j})$ are the same and in the both cases
they are finite sets. Therefore the set $\mathbb{S}(h)$ of indices
$k\in\mathbb{Z}$ for which $\frac{1}{\alpha_{k}}$ is nonintegrable in $B(h)$
is finite as the set $\mathbb{S}$ defined in (54) for the case $n=2\mu+1.$

2. In the case $n=2\mu$ the curve $l(\varepsilon,h)$ and set $D(\varepsilon
,h)$ have the same properties as the curve $l(\varepsilon)$ and set
$D(\varepsilon)$ in case $n=2\mu+1,$ where $D(\varepsilon,h)$ is the open set
enclosed by the curve $l(\varepsilon,h)\cup B(h).$ Namely, Proposition 3$(a)$
and $(b),$ which signify the reason of the simplicity of the odd case,
continue to hold if $2\mu+1,$ $l(\varepsilon)$, $D(\varepsilon)$ and (42) are
replaced by $2\mu,$ $l(\varepsilon,h)$, $D(\varepsilon,h)$ and (95) respectively.

Therefore taking into account that Proposition 2 and Theorem 1 hold for
$n=2\mu$ too and repeating the proof of Theorem 3 we obtain the following theorem.

\begin{theorem}
For each $f\in E$ the following equality
\begin{equation}
\int\limits_{l(\varepsilon,h)}f_{t}(x)dt=F(h)+\sum_{k\in\mathbb{Z}%
\backslash\mathbb{S}(h)}\int\limits_{B(h)}a_{k}(t)\Psi_{k,t}(x)dt
\end{equation}
hold, where $f_{t}(x)$ is defined by (15),
\[
F(h)=\int\limits_{B(h)}\sum_{k\in\mathbb{S}(h)}a_{k}(t)\Psi_{k,t}%
(x)dt=\lim_{\delta\rightarrow0}\sum_{k\in\mathbb{S}(h)}\int\limits_{I(\delta
,h)}a_{k}(t)\Psi_{k,t}(x)dt
\]
and $I(\delta,h)$ is obtained from $B(h)$ by deleting the $\delta
\in(0,\varepsilon)$ neighborhood of the SQ $t_{1\text{ }},t_{1\text{ }%
},...,t_{m\text{ }}$ lying in $B(h)$:%
\[
I_{\delta}(\delta,h)=B(h)\backslash%
{\textstyle\bigcup\nolimits_{j=1}^{m}}
(t_{j}-\delta,t_{j}+\delta).
\]
The series in (99) converge in the norm of $L_{2}(a,b)$ for every
$a,b\in\mathbb{R}.$
\end{theorem}

Now we consider the integral over $\gamma(0,h)$ in (98). The consideration of
integral over $\gamma(\pi,h)$ is similar. The complexity of the investigations
of the integral over $\gamma(0,h)$ is the following. In general, in the domain
enclosed by $\gamma(0,h)$ and $(-h,h)$ there exists infinite number of points
of $A.$ Moreover the interval $(-h,h)$ may contain, infinite number of SQ
defined in Definition 3. That is why, in the big contrary of the case $B(h)$
the set of indices $k$ for which $\frac{1}{\alpha_{k}}$ is nonintegrable over
$(-h,h)$ is not finite, in general, and may coincide with $\mathbb{Z}.$ This
situation complicate to replace $\gamma(0,h)$ with $[-h,h]$. Fortunately, now
we prove that, if we huddle together the terms $a_{k}(t)\Psi_{k,t}$ and
$a_{-k}(t)\Psi_{-k,t}$ for large value of $k$ then
\begin{equation}
\int\limits_{\gamma(0,h)}\left(  a_{k}(t)\Psi_{k,t}(x)+a_{-k}(t)\Psi
_{-k,t}(x)\right)  dt=\int\limits_{[-h,h]}\left(  a_{k}(t)\Psi_{k,t}%
(x)+a_{-k}(t)\Psi_{-k,t}(x)\right)  dt.
\end{equation}
Indeed, instead of the circle $C$ used in (32)\ taking the circle%
\[
\left\{  z\in\mathbb{C}:\left\vert z-(i2k\pi)^{n}\right\vert =\tfrac{1}%
{3}n(2\pi k+\tfrac{1}{4})^{n-1}\right\}  ,
\]
using Theorem 4 and repeating the proof of Theorem 1 we obtain

\begin{proposition}
For each $k>N_{h}(0)$ the equality (100) holds.
\end{proposition}

Similarly, in the case $\gamma(\pi,h)$, instead of $C$ taking the circle
\[
\left\{  z\in\mathbb{C}:\left\vert z-(i(2k\pi+\pi))^{n}\right\vert =\tfrac
{1}{3}n(2\pi k+\pi+\tfrac{1}{4})^{n-1}\right\}  ,
\]
using Theorem 5 and repeating the proof of Theorem 1 we obtain

\begin{proposition}
For each $k>N_{h}(\pi)$ the following equality holds.
\begin{equation}
\int\limits_{\gamma(\pi,h)}a_{k}(t)\Psi_{k,t}+a_{-(k+1)}(t)\Psi_{-(k+1),t}%
dt=\int\limits_{[\pi-h,\pi+h]}a_{k}(t)\Psi_{k,t}+a_{-(k+1)}(t)\Psi
_{-(k+1),t}dt.
\end{equation}

\end{proposition}

Now let us consider the terms $a_{k}(t)\Psi_{k,t}$ corresponding to the
eigenvalues (76). In Remark 5 we proved that there exist a closed curve
$\Gamma(0)$ which contains inside only the eigenvalues (76) of $L_{t}$ for
$|t|\leq h$. Instead of the curve $C$ using $\Gamma(0)$ and repeating the
proof of Theorem 1 we obtain
\begin{equation}
\sum_{k\in\mathbb{N}(0)}\int\limits_{\gamma(0,h)}a_{k}(t)\Psi_{k,t}%
(x)dt=\int\limits_{[-h,h]}\sum_{k\in\mathbb{N}(0)}a_{k}(t)\Psi_{k,t}(x)dt.
\end{equation}
Instead of (76) using (77) in the same way we get
\begin{equation}
\sum_{k\in\mathbb{N}(\pi)}\int\limits_{\gamma(\pi,h)}a_{k}(t)\Psi
_{k,t}(x)dt=\int\limits_{[\pi-h,\pi+h]}\sum_{k\in\mathbb{N}(\pi)}a_{k}%
(t)\Psi_{k,t}(x)dt.
\end{equation}

To consider the integrals on the right hand side of (100)-(103) we introduce
some notations. First let us consider the integrals (100) and (101). It
follows from (100) that if $k>N_{h}(0),$ then the expression $a_{k}%
(t)\Psi_{k,t}(x)$ is integrable over any subinterval $I$ of $(-h,h)$ if and
only if $a_{-k}(t)\Psi_{-k,t}(x)$ is integrable over $I$. Therefore, by
Theorem 9, the set of all SQ $t\in(-h,h)$ such that $\frac{1}{\alpha_{k}}$ is
nonintegrable in any neighborhood of $t$ coincides with the set of all SQ
$t\in(-h,h)$ such that $\frac{1}{\alpha_{-k}}$ is nonintegrable. For
$k>N_{h}(0)$ denote by $V(k,h,0)$ the set of all SQ $t\in(-h,h)$ for which
$\frac{1}{\alpha_{k}}$ and hence $\frac{1}{\alpha_{-k}}$ is nonintegrable in
any neighborhood of those SQ. We say that $V(k,h,0)$ corresponds to $\left\{
k,-k\right\}  .$ In the same way for $k>N_{h}(\pi)$ we construct the set
$V(k,h,\pi)$ of SQ $t\in(\pi-h,\pi+h)$ corresponding to $\left\{
k,-(k+1)\right\}  .$ It is clear that we should handle the terms $a_{k}%
(t)\Psi_{k,t}(x)$ and $a_{-k}(t)\Psi_{-k,t}(x)$ if and only if $V(k,h,0)$ is
not an empty set.

To determine the handling terms corresponding to the indices of $\mathbb{N}%
(0)$ and $\mathbb{N}(\pi)$ (see (76) and (77)) we introduce the following
notations. Denote by $\mathbb{S}(h,0)$ and $\mathbb{S}(h,\pi)$ respectively
the sets of indices $k\in\mathbb{N}(0)$ and $k\in\mathbb{N}(\pi)$ for which
$\frac{1}{\alpha_{k}}$ is nonintegrable in $(-h,h)$ and $(\pi-h,\pi+h).$ It is
clear that $k\in\mathbb{S}(h,0)$ if and only if $k\in\mathbb{N}(0)$ and there
exists SQ $t\in(-h,h)$ such that $\frac{1}{\alpha_{k}}$ is nonintegrable in
any neighborhood of $t.$ Let $V(h,0)$ be the set of all SQ such that for each
$t\in V(h,0)$ there exists $k\in\mathbb{S}(h,0)\subset\mathbb{N}(0)$ for which
$\frac{1}{\alpha_{k}}$ is nonintegrable over any neighborhood of $t.$ In the
same way we define $V(h,\pi)$ for $\mathbb{S}(h,\pi).$

\begin{proposition}
The number of elements of the sets $\mathbb{S}(h,0),$ $\mathbb{S}(h,\pi),$
$V(h,0),V(k,h,0),$ $V(h,\pi)$ and $V(k,h,\pi)$ are finite.
\end{proposition}

\begin{proof}
The sets $\mathbb{S}(h,0)$ and $\mathbb{S}(h,\pi)$ are finite, since
$\mathbb{S}(h,0)\subset\mathbb{N}(0),$ $\mathbb{S}(h,\pi)\subset\mathbb{N}%
(\pi)$ and the sets $\mathbb{N}(0)$ and $\mathbb{N}(\pi)$ are finite. By
Definition 3, if$\ t\in V(h,0)$ then there exists $k\in\mathbb{N}(0)$ such
that $\lambda_{k}(t)$ is an ESS and hence is a multiple Bloch eigenvalue lying
inside of the closed curve $\Gamma(0)$ defined in Remark 5. On the other hand
by Remark 1 the number of the multiple eigenvalues lying in a bounded domain
is finite. Moreover for each multiple eigenvalue $a_{k}$ there exist at most
$n$ values of $t\in Q$ satisfying $\Delta(a_{k},t)=0$ (see (8)). Therefore
$V(h,0)$ is finite. In the same way we prove that the sets $V(k,h,0)$,
$V(h,\pi)$ and $V(k,h,\pi)$ are finite
\end{proof}

Now using (21) for $l=\gamma(0,h)$ and $l=\gamma(0,h),$ taking into account
(100)-(103) and arguing as in the proof of Theorem 3 we obtain

\begin{theorem}
Let $f\in E$. \ Then the following equality
\begin{equation}
\int\limits_{\gamma(0,h)}f_{t}(x)dt=F(0,h)+\sum_{k\in\mathbb{N}(0)\backslash
\mathbb{S}(h,0)}\int\limits_{(-h,h)}a_{k}(t)\Psi_{k,t}(x)dt+
\end{equation}%
\[
\sum_{k>N_{h}(0)}\int\nolimits_{[-h,h]}\left(  a_{k}(t)\Psi_{k,t}%
(x)dt+a_{-k}(t)\Psi_{-k,t}(x)\right)  dt
\]
holds, where%
\[
F(0,h)=\int\limits_{(-h,h)}\sum_{k\in\mathbb{S}(h,0)}a_{k}(t)\Psi
_{k,t}(x)dt=\lim_{\delta\rightarrow0}\sum_{k\in\mathbb{S}(h,0)}\int
\limits_{I(\delta,0)}a_{k}(t)\Psi_{k,t}(x)dt
\]
and $I(\delta,0)$ is obtained from $[-h,h]$ by deleting $\delta$-neighborhoods
of the SQ lying in $V(h,0).$ Moreover for $\left\vert k\right\vert >N_{h}(0)$
we have
\begin{equation}
\int\nolimits_{\lbrack-h,h]}\left(  a_{k}(t)\Psi_{k,t}(x)dt+a_{-k}%
(t)\Psi_{-k,t}(x)\right)  dt=
\end{equation}%
\[
\lim_{\delta\rightarrow0}\left(  \int\nolimits_{I(k,\delta,0)}a_{k}%
(t)\Psi_{k,t}(x)dt+\int\nolimits_{I(k,\delta,0)}a_{-k}(t)\Psi_{-k,t}%
(x)dt\right)  ,
\]
where $I(k,\delta,0)$ is obtained from $[-h,h]$ by deleting $\delta
$-neighborhoods of the SQ lying in $V(k,h,0).$ The series in (104) converge in
the norm of $L_{2}(a,b)$ for every $a,b\in\mathbb{R}.$
\end{theorem}

\begin{theorem}
Let $f\in E.$ Then the following equality
\begin{equation}
\int\limits_{\gamma(\pi,h)}f_{t}(x)dt=F(\pi,h)+\sum_{k\in\mathbb{N}%
(\pi)\backslash\mathbb{S}(h,\pi)}\int\limits_{[\pi-h,\pi+h]}a_{k}(t)\Psi
_{k,t}(x)dt+
\end{equation}%
\[
\sum_{k>N_{h}(\pi)}\int\nolimits_{[\pi-h,\pi+h]}\left(  a_{k}(t)\Psi
_{k,t}(x)dt+a_{-(k+1)}(t)\Psi_{-(k+1),t}(x)\right)  dt
\]
holds, where%
\[
F(\pi,h)=\int\limits_{(\pi-h,\pi+h)}\sum_{k\in\mathbb{S}(h,\pi)}a_{k}%
(t)\Psi_{k,t}(x)dt=\lim_{\delta\rightarrow0}\sum_{k\in\mathbb{S}(h,\pi)}%
\int\limits_{I(\delta,\pi)}a_{k}(t)\Psi_{k,t}(x)dt
\]
and $I(\delta,h)$ is obtained from $[\pi-h,\pi+h]$ by deleting $\delta
$-neighborhoods of the SQ lying in $V(h,\pi).$ Moreover, for $\left\vert
k\right\vert >N_{h}(\pi)$ we have
\begin{equation}
\int\nolimits_{\lbrack\pi-h,\pi+h]}\left(  a_{k}(t)\Psi_{k,t}(x)dt+a_{-(k+1)}%
(t)\Psi_{-(k+1),t}(x)\right)  dt=
\end{equation}%
\[
\lim_{\delta\rightarrow0}\left(  \int\nolimits_{I(k,\delta,\pi)}a_{k}%
(t)\Psi_{k,t}(x)dt+\int\nolimits_{I(k,\delta,\pi)}a_{-(k+1)}(t)\Psi
_{-(k+1),t}(x)dt\right)  ,
\]
where $I(k,\delta,\pi)$ is obtained from $[\pi-h,\pi+h]$ by deleting $\delta
$-neighborhoods of the SQ lying in $V(k,h,\pi).$ The series in (106) converge
in the norm of $L_{2}(a,b)$ for every $a,b\in\mathbb{R}.$
\end{theorem}

Now using (98) and Theorems 10-12 we obtain

\begin{theorem}
For each $f\in E$ the following equality%
\begin{equation}
2\pi f=F(h)+F(0,h)+F(\pi,h)+\sum_{k\in\mathbb{Z}\backslash\mathbb{S}(h)}%
\int_{B(h)}a_{k}(t)\Psi_{k,t}dt+
\end{equation}%
\[
\sum_{k\in\mathbb{N}(0)\backslash\mathbb{S}(h,0)}\int\nolimits_{[-h,h]}%
a_{k}(t)\Psi_{k,t}dt+\sum_{k>N_{h}(0)}\int\nolimits_{[-h,h]}\left[
a_{k}(t)\Psi_{k,t}+a_{-k}(t)\Psi_{-k,t}\right]  dt+
\]%
\[
\sum_{k\in\mathbb{N}(\pi)\backslash\mathbb{S}(h,\pi)}\int\nolimits_{[\pi
-h,\pi+h]}a_{k}(t)\Psi_{k,t}dt+\sum_{k>N_{h}(\pi)}\int\nolimits_{[\pi
-h,\pi+h]}\left[  a_{k}(t)\Psi_{k,t}+a_{-(k+1)}(t)\Psi_{-(k+1),t}\right]  dt
\]
holds. All series in (108) converge in the norm of $L_{2}(a,b)$ for every
$a,b\in\mathbb{R}.$
\end{theorem}

Now we prove that the expansion (108) has the elegant form (26) if and only if
$L$ has no ESS and ESS at infinity. For this we prove the following lemma by
using Theorems 6-8.

\begin{lemma}
$(a)$ If $f\in E,$ then the equality
\[
\lim_{k\rightarrow\infty}(f_{t},\Psi_{k,t}^{\ast})=0
\]
holds uniformly with respect to $t\in\left(  \lbrack0,2\pi\right)  \backslash
A).$

$(b)$ For $t\in\lbrack-h,h]$ and $t\in\lbrack\pi-h,\pi+h]$ respectively the
inequalities
\[
\left\vert u_{k,k}(t)u_{k,k}^{\ast}(t)\right\vert +\left\vert u_{k,k}%
(t)u_{k,-k}^{\ast}(t)\right\vert +\left\vert u_{k,-k}(t)u_{k,k}^{\ast
}(t)\right\vert +\left\vert u_{k,-k}(t)u_{k,-k}^{\ast}(t)\right\vert >1/4
\]
and
\[
\left\vert u_{k,k}(t)u_{k,k}^{\ast}(t)\right\vert +\left\vert u_{k,k}%
(t)u_{k,-k-1}^{\ast}(t)\right\vert +\left\vert u_{k,-k-1}(t)u_{k,k}^{\ast
}(t)\right\vert +\left\vert u_{k,-k-1}(t)u_{k,-k-1}^{\ast}(t)\right\vert >1/4
\]
hold for the large values of $k.$
\end{lemma}

\begin{proof}
$(a)$ By (2) Theorem 8 continues to hold for the normalized eigenfunction
$\Psi_{k,t}^{\ast}$ of $L_{t}^{\ast}$ and hence the following uniform with
respect to $t\in\left(  \lbrack0,2\pi\right)  \backslash A)$ formula
\begin{equation}
\Psi_{k,t}^{\ast}(x)=e^{itx}(u_{k,k}^{\ast}(t)e^{i2\pi kx}+u_{k,-k}^{\ast
}(t)e^{-i2\pi kx}+u_{k,-(k+1)}^{\ast}(t)e^{-i2\pi(k+1)x}+\varphi_{k,t}^{\ast
}(x)),
\end{equation}
where $u_{k,m}^{\ast}(t)=\left(  \Psi_{k,t}^{\ast}(x),e^{i(2\pi m+t)x}\right)
,$ $\sup\nolimits_{x\in\lbrack0,1]}|\varphi_{k,t}^{\ast}(x)|=O\left(
k^{-1}\ln\left\vert k\right\vert \right)  ,$ holds. Since%
\[
(f_{t},\Psi_{k,t}^{\ast})=\int_{-\infty}^{\infty}f(x)\overline{\Psi
_{k,t}^{\ast}(x)}dx
\]
(see (19)), by (109), it is enough to prove that the following four integrals
approach zero uniformly with respect to $t\in\left(  \lbrack0,2\pi\right)  $
as $k\rightarrow\infty.$
\[
\int_{\mathbb{R}}f(x)e^{-i(2\pi k+t)x}dx,\text{ }\int_{\mathbb{R}%
}f(x)e^{i(2\pi k+t)x}dx,\text{ }\int_{\mathbb{R}}f(x)e^{i(2\pi(k+1)+t)x}%
dx,\text{ }\int_{[0,1]}f_{t}(x)\overline{\varphi_{k,t}^{\ast}(x)}dx.
\]
The first three integrals converge uniformly to zero as $k\rightarrow\infty,$
since the Fourier transform%
\[
\int_{\mathbb{R}}f(x)e^{-i\lambda x}dx
\]
of $f\in E$ (see Condition 1$(ii)$) converge to zero as $\lambda
\rightarrow\infty.$ The convergence of the Fourth integral follows from
Condition 1$(i)$ and the estimation for $\varphi_{k,t}^{\ast}(x)$ given in (109).

The proof of $(b)$ follows from Theorem 6 and Theorem 7
\end{proof}

\begin{theorem}
Let $n=2\mu$ and $L$ has no ESS at infinity. Then

$(a)$ There exist a positive constants $M$ and $N$ such that%
\begin{equation}
\int_{\lbrack0,2\pi)}\left\vert \left(  \alpha_{k}(t)\right)  ^{-1}\right\vert
dt<M,\text{ }\forall\left\vert k\right\vert >N.
\end{equation}

$(b)$ The number of SQ and ESS is finite. The set $\mathbb{S}$ of all $k$ for
which $\frac{1}{\alpha_{k}}$ \ is nonintegrable in $[0,2\pi)$ is finite.

$(c)$ For each $f\in E$ the following spectral expansion holds
\begin{equation}
2\pi f(x)=\int\limits_{[0,2\pi)}\sum\limits_{k\in\mathbb{S}}a_{k}(t)\Psi
_{k,t}(x)dt+\sum_{k\in\mathbb{Z}\backslash\mathbb{S}}\int\limits_{[0,2\pi
)}a_{k}(t)\Psi_{k,t}(x)dt,
\end{equation}
where
\begin{equation}
\int\limits_{\lbrack0,2\pi)}\sum\limits_{k\in\mathbb{S}}a_{k}(t)\Psi
_{k,t}(x)dt=\lim_{\delta\rightarrow0}\sum_{k\in\mathbb{S}}\int
\nolimits_{I(\delta)}a_{k}(t)\Psi_{k,t}dt
\end{equation}
and $I_{\delta}$ is obtained from $[0,2\pi)$ by deleting $\delta
$-neighborhoods of the SQ. The series in (111) converge in the norm of
$L_{2}(a,b)$ for every $a,b\in\mathbb{R}.$
\end{theorem}

\begin{proof}
The proof of $(a)$ follows Definition 4. $(b)$ follows from $(a).$ Now let us
prove $(c).$ For this we show that
\begin{equation}
\int\nolimits_{\lbrack0,2\pi)}a_{k}(t)\Psi_{k,t}(x)dt\text{ }%
\end{equation}
exist for $k\in\mathbb{Z}\backslash\mathbb{S}$ and its $L_{2}(-p,p)$ norm
tends to zero as $\left\vert k\right\vert \rightarrow\infty.$

The existence of (113) follows from (110) and Theorem 9. Using (110) and Lemma
3 and taking into account that \ $\left\vert \Psi_{k,t}(x)\right\vert \leq4$
for $k>N$ and $x\in\lbrack-p,p]$ (see (83)) we obtain that (113) converge to
zero as $k\rightarrow\infty$ uniformly with respect to $x\in\lbrack-p,p].$
Hence $L_{2}(-p,p)$ norm of (113) tends to zero as $k\rightarrow\infty.$ It
means that, in the contrary to the general case we need not to huddle together
the terms $a_{k}(t)\Psi_{k,t}(x)$ and $a_{-k}(t)\Psi_{-k,t}(x)$ ($a_{k}%
(t)\Psi_{k,t}$ and $a_{-(k+1)}(t)\Psi_{-(k+1),t}$ ) in the case of integration
over $[-h,h]$ ($[\pi-h,\pi+h]$) (see (108)). Hence the proof follows from
$(a),$ $(b)$ and Theorem 13
\end{proof}

Finally we prove the following main result.

\begin{theorem}
If $L$ has no ESS and ESS at infinity then for each $f\in E$ (26) holds and
the series in (26) converge in the norm of $L_{2}(a,b)$ for every
$a,b\in\mathbb{R}.$ Conversely, if $L$ has ESS or ESS at infinity, then there
exists $f\in L_{2}(-\infty,\infty)$ such that the spectral decomposition (26)
does not hold.
\end{theorem}

\begin{proof}
By Theorem 14 if $L$ has no ESS at infinity then (111) holds. If, in addition,
$L$ has no ESS then the set $\mathbb{S}$ is an empty set. Therefore (26)
follows from (111).

Now suppose that $L$ has ESS or ESS at infinity. If $L$ has ESS then is
follows from Theorem 13 that the spectral expansion has no form (26). Now
consider the case when $L$ has no ESS and hence has only ESS at infinity. Then
by Theorem 14 $(a)$ there exist sequence of integers $k_{s}$ such that the
following integral exists and
\begin{equation}
\int_{\lbrack0,2\pi)}\left\vert \left(  \alpha_{k_{s}}(t)\right)
^{-1}\right\vert dt\rightarrow\infty
\end{equation}
as $s\rightarrow\infty.$ Due to (2), $\Psi_{k,t}^{\ast}$ \ satisfy (7). It
with (7) implies that (45) holds uniformly in $[h,\pi-h]\cup\lbrack\pi
+h,2\pi-h]$. Therefore (114) holds if $\ [0,2\pi)$ is replaced by
$[0,h]\cup\lbrack\pi-h,\pi+h]\cup\lbrack2\pi-h,2\pi)$. It implies that,
without loss of generality, it can be assumed that there exist sequence of
integers, denoted for simplicity of notations again by $k_{s},$ such that
\begin{equation}
\lim_{s\rightarrow\infty}\int_{[0,h]}\left\vert \left(  \alpha_{k_{s}%
}(t)\right)  ^{-1}\right\vert dt=\infty,
\end{equation}
where $h<1/32$ (see (67)).

Now we prove that if (115) holds then there exists $f\in L_{2}(-\infty
,\infty)$ and sequence of integers, denoted for simplicity of notations again
by $k_{s},$ such that
\begin{equation}
\left\Vert \int_{\lbrack0,2\pi)}a_{k_{s}}(t)\Psi_{k_{s},t}(x)dt\right\Vert
\geq\frac{1}{2}%
\end{equation}
for large values of $s.$ Using Lemma 3$(b)$ and (115) we obtain that
\begin{equation}
\lim_{s\rightarrow\infty}\left(  \max_{j,m\in\left\{  -k_{s},k_{s}\right\}
}\int\nolimits_{[0,h]}\left\vert \left(  \alpha_{k_{s}}(t)\right)  ^{-1}%
u_{j}(t)u_{m}^{\ast}(t)\right\vert dt\right)  =\infty,
\end{equation}
where $u_{k_{s},j}(t)$ and $u_{k_{s},m}^{\ast}(t)$ are re-denoted by
$u_{j}(t)$ and $u_{m}^{\ast}(t)$ respectively. It is clear that there exists
subsequence of $\left\{  k_{s}\right\}  $, denoted for simplicity of notation
again by $\left\{  k_{s}\right\}  $, such that maximum in (117) gets for a
fixed pair $\left(  j,m\right)  $. Suppose without loss of generality that the
maximum gets for $j=m=k_{s}.$ Then
\begin{equation}
\lim_{s\rightarrow\infty}\int\nolimits_{[0,h]}\left\vert \left(  \alpha
_{k_{s}}(t)\right)  ^{-1}u_{k_{s}}(t)u_{k_{s}}^{\ast}(t)\right\vert dt=\infty.
\end{equation}
Let $b(s)$ denotes the integral in (118). Since $b(s)\rightarrow\infty$ as
$s\rightarrow\infty,$ the sequence $\left\{  k_{s}\right\}  $ can be chosen so
that
\begin{equation}
k_{s}>s^{2},\text{ }b(s)>s^{2}%
\end{equation}
for all $\left\vert s\right\vert >N,$ where $N$ is a large number. Define $f$
by formula
\begin{equation}
f_{t}(x)=e^{itx}\sum\limits_{s>N}c_{s}(t)\left(  b(s)\right)  ^{-1}e^{i2\pi
k_{s}x},
\end{equation}
where $c_{s}$ is a measurable function such that $\left\vert c_{s}%
(t)\right\vert =1$ for all $t\in\lbrack0,h]$ and $c_{s}(t)=0$ for all
$t\in\lbrack0,2\pi)\backslash\lbrack0,h].$ It follows from (119) that the
function defined by (120) belongs to $L_{2}([0,1]\times\lbrack0,2\pi)).$ Then
$f$ defined by (14) belongs to $L_{2}(-\infty,\infty)$. Moreover, $c_{s}(t)$
can be chosen so that
\[
c_{s}(t)\left(  b(s)\right)  ^{-1}\left(  \overline{\alpha_{k_{s}}(t)}\right)
^{-1}\overline{u_{k_{s}}^{\ast}(t)}u_{k_{s}}(t)=\left(  b(s)\right)
^{-1}\left\vert \left(  \alpha_{k_{s}}(t)\right)  ^{-1}u_{k_{s}}^{\ast
}(t)u_{k_{s}}(t)\right\vert ,\text{ }\forall t\in\lbrack0,h].
\]
Therefore using the definition of $b(s)$ we obtain%
\begin{equation}
\int\nolimits_{\lbrack0,2\pi)}\left(  \overline{\alpha_{k_{s}}(t)}\right)
^{-1}(f_{t},\Psi_{k_{s},t}^{\ast})u_{k_{s}}(t)dt=\left(  b(s)\right)
^{-1}\int\nolimits_{[0,h]}\left\vert \left(  \alpha_{k_{s}}(t)\right)
^{-1}u_{k_{s}}(t)u_{k_{s}}^{\ast}(t)\right\vert dt=1.
\end{equation}
\ \ Now using (121), Theorem 6 and the obvious inequality
\begin{equation}
e^{itx}=1+c(t,x),\text{ }\left\vert c(t,x)\right\vert <2h<1/16
\end{equation}
for all $t\in\lbrack0,h]$ and $x\in\lbrack0,1]$ we obtain that
\begin{equation}
\int\nolimits_{\lbrack0,h]}a_{k_{s}}(t)\Psi_{k_{s},t}dt=I_{k_{s}}+I_{-k_{s}%
}+C_{k_{s}}+C_{-k_{s}}+O(k_{s}^{-1}\ln\left\vert k_{s}\right\vert ),
\end{equation}
where $a_{k_{s}}(t)=\left(  \overline{\alpha_{k_{s}}(t)}\right)  ^{-1}%
(f_{t},\Psi_{k_{s},t}^{\ast})=\left(  b(s)\right)  ^{-1}\left(  \overline
{\alpha_{k_{s}}(t)}\right)  ^{-1}\overline{u_{k_{s}}^{\ast}(t)}c_{s}(t),$%
\[
I_{\pm k_{s}}=\left(  \int\nolimits_{[0,h)]}a_{k_{s}}(t)u_{\pm k_{s}%
}(t)dt\right)  e^{\pm i2\pi k_{s}x}\text{ , }C_{\pm k_{s}}=\left(
\int\nolimits_{[0,h)]}a_{k_{s}}(t)u_{\pm k_{s}}(t)c(t,x)dt\right)  e^{\pm
i2\pi k_{s}x}\text{ .}%
\]
By (121) and (122)\ the following relations hold
\[
\left\Vert C_{\pm k_{s}}\right\Vert <1/16,\text{ }\left(  I_{k_{s}},I_{-k_{s}%
}\right)  =0,\text{ }\left\Vert I_{k_{s}}\right\Vert =1.
\]
Therefore from (123) we obtain that for the chosen $f$ the relation (116)
holds. It implies that the series in (26) does not converge in norm of
$L_{2}(-p,p),$ for any integer $p$
\end{proof}

Note that, since $E$ contain the set of all continuous and compactly supported
functions which is a dense set in $L_{2}(-\infty,\infty),$ one can find $f\in
E$ for which (26) does not hold. Since it is very technical we do not give the
corresponding calculations in this paper.

\begin{conclusion}
In Remark 4 we discussed the necessity of the parenthesis comprising the
eigenfunctions corresponding to the eigenvalue which is ESS. In Theorem 15 we
proved that if the operator $L$ has ESS at infinity then the series, in
general, does not converges without parenthesis. The necessity of the
parenthesis for the convergence of the series taking parts in the spectral
decomposition (see (108)) can also be explained as follows. The following
cases are possible.

Case 1. The operator $L$ has infinitely many ESS. Then as we explained in
Remark 4 the parenthesis for each ESS is necessary and hence we will have the
infinitely many parenthesis and hence the spectral expansion series converge
with parenthesis.

Case 2. The number of ESS is at most finite and $L$ has ESS at infinity that
is (25) holds. It is clear that the series of spectral expansion for $f$
converges without parenthesis if and only if (113) converges to zero as
$k\rightarrow\infty$. However, for the consideration of the limit of (113),
due to (23), we meet with the indeterminate form $\infty\cdot0,$ since (25)
holds and $\left(  f_{t},\Psi_{k,t}^{\ast}\right)  \rightarrow0$ as
$k\rightarrow\infty$ because of Lemma 3. In fact, (113) may converges to zero
if $\left(  f_{t},\Psi_{k,t}^{\ast}\right)  \rightarrow0$ more rapidly than
the integral in (25) tends to infinity as $s\rightarrow\infty$. The converse
statement is also possible. It depends on $f$ and behavior of $\alpha_{k}(t)$
near to $0$ and $\pi.$

Thus the nature of the spectral expansion problem requires to use the
parenthesis and the complicated form of the spectral expansion theorem is
connected with a complicated picture of the spectrum and projections of the
operators with complex-valued periodic coefficients. Moreover, one can obtain
a spectral expansion without parenthesis if and only if $L$ has no ESS (or
equivalently has no SQ) and ESS at infinity (see the last theorem). It implies
that, in the general case, it is necessary to use the parenthesis, since the
operator $L$ has ESS and ESS at infinity.
\end{conclusion}

\end{document}